\documentclass[11pt,reqno]{amsart}
\usepackage{amsmath, amsfonts, amsthm, amssymb,amscd, graphicx, amscd}
\usepackage{float,epsf}
\usepackage[english]{babel}
\usepackage{enumerate}
\usepackage{tikz}
\usepackage{mathrsfs}
\usepackage[numbers,sort&compress]{natbib}

\usepackage{srcltx}
\usepackage{geometry}
\usepackage{verbatim}
\usepackage{mathrsfs}
\usepackage{hyperref}
\usepackage{enumitem} % C: Added for different list labels
\usepackage{bbm} % C: For blackboard numbers
\usepackage{stmaryrd}

\usepackage{relsize}
\usepackage{exscale}

\usepackage{mathtools}

\newtheorem{thm}{Theorem}[section]

\newtheorem{prop}[thm]{Proposition}
\newtheorem{lem}[thm]{Lemma}

\newtheorem{rem}[thm]{Remark}

\numberwithin{equation}{section}

\def\dd{{\rm d}}
\hypersetup{bookmarksdepth=2}

\begin{document}
\title{The relativistic Euler-Poisson equation as a mean-field limit}
\author{Immanuel Ben-Porat}
\begin{abstract}
We adapt the renormalized energy method developed in \cite{duerinckx2020mean} in order to prove the monokinetic mean field limit for relativistic Newtonian dynamics. The resulting monokinetic PDE is the relativistic Euler-Poisson equation. Since the position evolves relativistically in comparision to the momentum, the kinetic part of the modulated energy has to adjusted, leading to further obstructions that are not present in the non-relativistic settings. A weak-strong stability principle for the relativistic Euler-Poisson equation is investigated, followed by a renormalization procedure leading to the mean-field limit. Our main result constitute the first mean-field limit for relativistic Coulomb flows.     
\end{abstract}
\maketitle
\section{Introduction}
\subsection{Main Objective.} Consider the relativistic $N$-body dynamics 
\begin{align}\tag{$N$r}
\begin{cases}
\begin{array}{lc}
     \dot X_{i}^
     {N}(t)=v(P_{i}^{N}(t)),\ X_{i}^{N}(0)=X_{i}^{0,N}  \\
     \dot P_{i}^{N}(t)=-\frac{1}{N}\sum_{j:j\neq i}\nabla V(X_{i}^{N}(t)-X_{j}^{N}(t)),\ P_{i}^{N}(0)=P_{i}^{0,N}.  
\end{array}
\end{cases}
\label{Relativistic N body dynamics}
\end{align}
The pair $(X_{i},P_{i})$ represents the position and momentum of the $i$-th particle, the function $v$ is the relativistic velocity given by 
\begin{align}
v(P)=\frac{P}{\sqrt{1+\left\vert P\right\vert^{2}}} \label{def of v}    
\end{align}
and $V$ is the $d$-dimensional ($d\geq 2$) periodic repulsive Coulomb interaction, i.e. the unique solution of 
\begin{align*}
-\Delta V=\delta_{0}-1, \ \int_{\mathbb{T}^{d}}V(x)\ \dd x=0.     
\end{align*}
Our ultimate goal is to prove the monokinetic mean-field limit of \eqref{Relativistic N body dynamics}, i.e. to derive the \textit{relativistic pressureless Euler-Poisson equation}:
\begin{align}\tag{EPr}
\begin{cases}
\begin{array}{lc}
\partial_{t}\rho+\mathrm{div}(\rho v(p))=0, \ \rho(0,\cdot)=\rho^{0}
       \\
\partial_{t}p+v(p)D_{x}p=-\nabla V\ast \rho,\ p(0,\cdot)=p^{0}.        
\end{array}  
\end{cases}
\label{Relativistic Euler intro}
\end{align}  
The unknowns are a time-dependent probability density $\rho(t,\cdot)\in \mathcal{P}(\mathbb{T}^{d})$ and a time-dependent momentum field $p(t,\cdot):\mathbb{T}^{d}\rightarrow \mathbb{R}^{d}$. The convention for the definition of the Jacobian is $(D_{x}p)_{ij}=(\frac{\partial p^{j}}{\partial x_{i}})_{ij}$. 
This derivation is to be understood in terms of the spatial empirical measure $\rho_{N}$ and the empirical momentum $q_{N}$: Given a solution $(\mathbf{X}_{N}(t),\mathbf{P}_{N}(t))=\left(X_{1}(t),\dots,X_{N}(t),P_{1}(t),\dots, P_{N}(t)\right)$ to the ODE system \eqref{Relativistic N body dynamics} we define $\rho_{N}$ by 
\begin{align}
\rho_{N}(t,\dd x)=\frac{1}{N}\sum_{i=1}^{N}\delta_{x=X_{i}(t)} \label{spatial empirical measure def}    
\end{align}
and $q_{N}(t,\dd x)$ by 
\begin{align}
q_{N}(t,\dd x)=\frac{1}{N}\sum_{i=1}^{N}\delta_{x=X_{i}(t)}P_{i}(t). \label{empirical momentum def}  \end{align}
We seek a weak convergence of the form $\rho_{N}(t,\cdot)\underset{N\rightarrow \infty}{\rightharpoonup}\rho(t,\cdot)$ and $q_{N}(t,\cdot)\underset{N\rightarrow \infty}{\rightharpoonup} \rho p(t,\cdot)$. 
\subsection{Background} Approximating large particle systems by macroscopic dynamics is a longstanding problem going back to the seminal work of Braun-Hepp \cite{braun1977vlasov} and Dobrushin \cite{dobrushin1979vlasov}. In more detail, the work of Dobrushin 
proves that the Vlasov equation 
\begin{align}\tag{V}
\partial_{t}f+\xi\cdot\nabla_{x}f-\nabla V\ast \rho_{f}\cdot \nabla_{\xi}f=0, \ \rho_{f}(t,x)=\int_{\mathbb{R}^{d}}f(t,x,\xi)\ \dd \xi \label{Vlasov equation}
\end{align}
can be obtained as a mean-field limit from the $N$-body dynamics 
\begin{align}
\begin{cases}
\begin{array}{lc}
    \dot X_{i}^{N}(t)=\Xi_{i}^{N}(t), \ X_{i}^{N}(0)=X_{i}^{0,N} &  \\
     \dot \Xi_{i}^{N}(t)=-\frac{1}{N}\sum_{i=1}^{N}\nabla V(X_{i}^{N}(t)-X_{j}^{N}(t)), \ \Xi_{i}^{N}(0)=\Xi_{i}^{0,N}.&  
\end{array}
\end{cases}
\label{N body nonrel}
\end{align}
The mean-field limit is to be understood in term of the empirical measure $\mu_{N}(t,\dd x\dd \xi)\coloneqq \frac{1}{N}\sum_{i=1}^{N}\delta_{x=X_{i}(t)}\otimes \delta_{\xi=\Xi_{i}(t)}$ and the Wasserstein distance, i.e. one ultimately seeks a statement of the form: 
\begin{align*}
W_{1}(\mu_{N}(0),f(0,\cdot))\underset{N\rightarrow \infty}{\rightarrow} 0\Rightarrow \underset{t\in [0,T]}{\sup}W_{1}(\mu_{N}(t,\cdot),f(t,\cdot))\underset{N\rightarrow \infty}{\rightarrow}0.  \end{align*}
The influential work of Dobrushin has the limitation that it is valid only for interactions $V$ enjoying $C^{1,1}$ regularity, which excludes Coulombic singularities. When $V$ is Coulomb it is customary to refer to \eqref{Vlasov equation} as the Vlasov-Poisson equation. So, otherwise put the mean-field limit for the Vlasov-Poisson is well beyond the scope of Dobrushin's mean-field limit result.   
Since the work of Dobrushin there has been a large amount of literature devoted to the mean-field limit for interactions with Riesz/Coulomb singularity as well as interactions exhibiting mild regularity/roughness. The $1$-dimensional Vlasov-Poisson equation has been derived as a mean-field limit from \eqref{N body nonrel} in \cite{hauray2012mean}. In higher dimensions, the case of interactions exhibiting sub-Coulombic singularities has been addressed in \cite{hauray2007n,hauray2015particle}. Mean-field limits with cutoffs, i.e. with a $N$-regularization incorporated in \eqref{N body nonrel} which disappears as $N\rightarrow \infty$, have been proved in \cite{lazarovici2016vlasov,lazarovici2017mean,FeistlHeldPickl2025}.  Although the problem of deriving the Vlasov-Poisson equation as a non-cutoff mean-field limit remains open in arbitrary dimensions, there has been recent promising progress which resolves the $2D$ case \cite{bresch2025new,duerinckx2026derivation} and the short time viscous $3D$ case \cite{FengWang2026}. These works make extensive use of relative entropy techniques. The method of the relative entropy has also been successfully applied to study the mean-field limit for interactions exhibiting roughness, see \cite{JabinWang2018, jabin2016mean}. Somewhat surprisingly, we are not aware of any works which consider the mean-field limit for relativistic Vlasov equations with singular interactions, although we mention the important work \cite{golse2012mean} which studies the mean-field limit of a regularized version of the relativistic Vlasov-Maxwell equation. We also refer to \cite{jabin2014review} for an exhaustive overview of the mean-field limit for Vlasov equations. 

\vspace{0.4 cm}

The Vlasov-Poisson equation is a kinetic PDE, but it is also possible to study the mean-field limit for monokinetic PDEs, which are obtained from \eqref{Vlasov equation} by introducing appropriate averaged quantities. The pioneering works \cite{duerinckx2020mean,bresch2020modulated} follow this route, and provide the derivation of the pressureless Euler-Poisson equation as a mean-field limit by invoking the \textit{renormalized modulated energy}. We also refer to \cite{carrillo2021mean} for the application of this method in the context of an Euler-Poisson equation with  nonlocal alignment terms, to \cite{rosenzweig2022mean, HessChildsRosenzweigSerfaty2026} for the extension of this method to transport fields with lower than Lipschitz regularity and to \cite{de2023sharp} for a uniform in time mean-field limit via the renormalized energy method.  Considering that the renormalized energy method  will be of particular importance for the present work, we briefly review the key insights of \cite{duerinckx2020mean}. First, recall that the (non-relativistic) Euler-Poisson equation reads 
\begin{align} 
\begin{cases}
\begin{array}{lc}
\partial_{t}\rho+\mathrm{div}(\rho u)=0, \ \rho(0,\cdot)=\rho^{0}\\\partial_{t}u+uD_{x}u=-\nabla V\ast \rho , \ u(0,\cdot)=u^{0}.  
\end{array} 
\end{cases}
\label{Euler Poisson}
\end{align}
Moreover, recall that \eqref{Euler Poisson} is a monokinetic version of the Vlasov-Poisson equation in the following sense: 
\begin{align*}
 \mbox{The monokinetic ansatz $\rho(t,x)\delta(\xi-u(t,x))$ solves \eqref{Vlasov equation} $\iff$ $(\rho,u)$ solves \eqref{Euler Poisson}}. \end{align*}
The first observation which paves the path to the derivation of \eqref{Euler Poisson} as a mean-field limit is that \eqref{Euler Poisson} satisfies a \textit{weak-strong stability principle}. By this we mean as follows: given a pair of solutions $(\rho_{1},u_{1}),(\rho_{2},u_{2})$ to \eqref{Euler Poisson} we define the (non-renormalized) modulated energy by  
\begin{align}
H(t)=2\int_{\mathbb{T}^{d}}\left\vert u_{1}-u_{2}\right\vert^{2}(t,x)\rho_{1}(t,x)\ \dd x+\int_{\mathbb{T}^{d}}V\ast (\rho_{1}-\rho_{2})(t,x)(\rho_{1}-\rho_{2})(t,x)\ \dd x. \label{nonrel Modulated energy}      
\end{align}
Then, a Gr\"onwall estimate on $H(t)$ reveals that it verifies the estimate $H(t)\leq e^{Ct}H(0)$ for a constant $C=C(\left\Vert D_{x}u_{2}\right\Vert_{\infty},\left\Vert \rho_{2}\right\Vert_{L^{\infty} L^{q}})$ (for some $1\leq q<\infty$). Crucially,  the constant $C$ involves the regularity of only one of the solutions, which suggests that this argument might be well adapted to the mean-field limit. This motivates the definition of the renormalized modulated energy, denoted by $H_{N}(t)$ and defined for solutions $(\mathbf{X}_{N}(t),\mathbf{\Xi}_{N}(t))$ and $(\rho,u)$ of \eqref{N body nonrel} and \eqref{Euler Poisson} respectively by 
\begin{align}
H_{N}(t)&=\underset{\mbox{kinetic part}}{\underbrace{\frac{2}{N}\sum_{i=1}^{N}\left\vert \Xi_{i}(t)-u(t,X_{i}(t))\right\vert^{2}}}+\underset{\mbox{interaction part}}{\underbrace{\int_{\Delta^{c}}V(x-y)\left(\rho_{N}-\rho\right)^{\otimes 2}(t,\dd x\dd y)+\mathcal{C}_{N}}}. \label{nonrel renormalized modulated energy}
\end{align}
In the above the definition we invoked the following notation: 
\begin{itemize}
    \item $\Delta^{c}$  designates the complement of the diagonal, i.e. $\Delta^{c}\coloneqq\{(x,y)\in \mathbb{T}^{d}\times \mathbb{T}^{d}\vert x\neq y\}$. More generally we set $\Delta_{N}^{c}\coloneqq \{(X_{1},\dots,X_{N})\vert \forall i\neq j: X_{i}\neq X_{j}\}$.
    \item $\rho_{N}(t,\dd x)$ is the spatial empirical measure centered at $\mathbf{X}_{N}(t)$, as defined in \eqref{spatial empirical measure def}. 
    \item $\mathcal{C}_{N}$ is the constant defined in Proposition \ref{non negativity}.
\end{itemize}
A lengthy calculation using \eqref{Euler Poisson} and \eqref{N body nonrel} shows that the time derivative of $H_{N}(t)$ is given by 
\begin{align*}
\frac{\dd}{\dd t}H_{N}(t)=&-\frac{2}{N}\sum_{i=1}^{N}D_{x}u(t,X_{i}(t)):(u(t,X_{i}(t))-\Xi_{i}(t))\otimes (u(t,X_{i}(t))-\Xi_{i}(t))\\
&+\int_{\Delta^{c}}(u(t,x)-u(t,y))\cdot \nabla V(x-y)(\rho_{N}-\rho)^{\otimes 2}(t,\dd x\dd y)\coloneqq\mathcal{D}_{1}(t)+\mathcal{D}_{2}(t).
\end{align*}
The term $\mathcal{D}_{1}(t)$ is readily seen to be bounded in terms of the kinetic part. The difficult term is the term $\mathcal
{D}_{2}(t)$ involving the interaction. The main novelty of \cite{duerinckx2020mean} is that $\mathcal{D}_{2}(t)$ can be bounded by means of the interaction part up to a negligible correction in $N$. More precisely, one has the following functional inequality:  
\begin{prop}\label{Commutator estimates} \textup{(Proposition 2.2, \cite{duerinckx2020mean})} \label{commutator estimate} 
Let $\rho\in L^{\infty}(\mathbb{T}^{d})\cap \mathcal{P}(\mathbb{T}^{d})$ and  $\mathbf{X}_{N}=(X_{1},\dots,X_{N})\in \Delta_{N}^{c}$.  Assume further that $u:\mathbb{T}^{d}\rightarrow \mathbb{R}^{d}$ is Lipschitz. Set $\rho_{N}=\frac{1}{N}\sum_{i=1}^{N}\delta_{X_{i}}$ and 
\begin{align}
\mathcal{V}_{N}(\rho_{N},\rho)\coloneqq\int_{\Delta^{c}} V(x-y)(\rho_{N}-\rho)^{\otimes 2}(\dd x\dd y). \label{def of VN}   
\end{align}
Then,
it holds that
 \begin{align*}
\bigg\vert \int_{\Delta^{c}
}\left(u(x)-u(y)\right) \cdot\nabla V(x-y)\left(\rho_{N}-\rho\right)^{\otimes2}(\dd x\dd y)\bigg\vert \\
\leq C\Big(\mathcal{V}_{N}(\rho_{N},\rho)+\frac{\log N}{N}\mathbf{1}_{d=2}+\frac{1}{N^{\frac{2}{d}}}\mathbf{1}_{d\geq3}\Big),
\end{align*}
where $C=C\left(d,\left\Vert u \right\Vert_{W^{1,\infty}},\left\Vert \rho\right\Vert_{\infty} \right)$. 
\end{prop}
With the aid of Proposition of \ref{Commutator estimates} it is possible to bound $\mathcal{D}_{2}(t)$ by means of the interaction part, thereby implying a Gr\"onwall estimate on $H_{N}(t)$, which in turn proves the mean-field limit.  
\subsection{Main new results.} Since its discovery, the renormalized modulated energy has been successfully  applied in a variety of different scenarios, including mean-field quantum limits \cite{golse2012mean}, hydrodynamical-mean-field limits \cite{han2021newton, rosenzweig2023rigorous}, mean-field limit for spray models \cite{menard2024mean} and mean-field limits for adaptive dynamics \cite{ben2026singular}.
The main aim of this work is to demonstrate that the renormalized modulated energy   developed in \cite{duerinckx2020mean} can be employed in the context of \textit{relativistic Newtonian dynamics}, witnessing once again the robustness of the method. In order to be able to carry out properly this method in the relativistic settings several conceptual adjustments must be implemented, and therefore our main result goes beyond being merely a technical improvement -- as will be clarified in \ref{subsec novelty}. First, as already explained, the renormalized modulated energy is modeled after a weak-strong stability principle. Therefore one has to first understand how the relativistic effect is incorporated within the weak-strong stability argument.
Given a vector field $p:\mathbb{T}^{d}\rightarrow \mathbb{R}^{d}$ set
\begin{align}
 \mathscr{K}(p)\coloneqq\sqrt{1+\left\vert p\right\vert^{2}}\label{def of rel kinetic part}   
\end{align}
and given vector fields $p_{1},p_{2}:\mathbb{T}^{d}\rightarrow \mathbb{R}^{d}$ set 
\begin{align}
\eta(p_{1}\vert p_{2})\coloneqq\mathscr{K}(p_{1})-\mathscr{K}(p_{2})-\nabla \mathscr{K}(p_{2})\cdot (p_{1}-p_{2}). \label{definition of eta intro}     
\end{align}
We introduce the following time dependent quantity, which is a relativistic version of the modulated energy \eqref{nonrel Modulated energy}: 
given a pair $(\rho_{1},p_{2}),(\rho_{2},p_{2})$ of solutions to \eqref{Relativistic Euler intro} define  
\begin{align}
\mathcal{E}(t)\coloneqq &2\int_{\mathbb{T}^{d}}\eta(p_{1}(t,x)\vert p_{2}(t,x))\rho_{1}(t,x)\ \dd x \notag\\
&+\int_{\mathbb{T}^{d}}V\ast (\rho_{1}-\rho_{2})(t,x)(\rho_{1}-\rho_{2})(t,x)\ \dd x \coloneqq\mathcal{K}(t)+\mathcal{V}(t). \label{rel modulated energy intro}   
\end{align}
Note that if we take $\mathscr{K}(p)=\left\vert p\right\vert^{2}$ then we recover the previously introduced non-relativistic modulated energy. Our first task is to prove a weak-strong stability principle for the relativistic Euler-Poisson equation, by obtaining a Gr\"onwall estimate for the relativistic modulated energy $\mathcal{E}(t)$, as summarized in the following theorem.  
\begin{thm} \
\label{weak strong stability intro}
Let $(\rho_{1},p_{1})$, $(\rho_{2},p_{2})$ be  solutions of 
\begin{align*}
\begin{cases}
&\partial_{t}\rho_{i}+\mathrm{div}(\rho_{i}v(p_{i}))=0, \ \rho_{i}(0,\cdot)=\rho_{i}^{0}\\
&\partial_{t}p_{i}+v(p_{i})D_{x}p_{i}=-\nabla_{x}V\ast \rho_{i},\ p_{i}(0,\cdot)=p_{i}^{0}. 
\end{cases}      
\end{align*}
Suppose that $(\rho_{2},p_{2})\in L^{\infty}([0,T];L^{a}(\mathbb{T}^{d}))\times L^{\infty}([0,T];W^{1,\infty}(\mathbb{T}^{d}))$ for some $a>d$. 
Let $\mathcal{E}(t)$ be given by \eqref{rel modulated energy intro}. Then, it holds that 
\begin{align}
\frac{\dd}{\dd t}\mathcal{E}(t)=&2\int_{\mathbb{T}^{d}} v(p_{2})\cdot \nabla V\ast (\rho_{1}-\rho_{2})(\rho_{1}-\rho_{2})\ \dd x \notag\\
&+2\int_{\mathbb{T}^{d}}\rho_{1}(v(p_{1})-v(p_{2}))\cdot ((p_{1}-p_{2})D_{x}(v(p_{2})))\ \dd x \notag\\
&+2\int_{\mathbb{T}^{d}}\rho_{1}\nabla V\ast \rho_{2}\left(\nabla^{2}\mathscr{K}(p_{2})(p_{1}-p_{2})-(\nabla \mathscr{K}(p_{1})-\mathscr{\nabla}\mathscr{K}(p_{2}))\right)\ \dd x. 
\label{statement of formula for der of H}
\end{align}
Consequently, there is some $C=C(a,d,\left\Vert p_{2}\right\Vert_{L^{\infty}W^{1,\infty}}, \left\Vert \rho_{2}\right\Vert_{L^{\infty}L^{a}})>0$ such that 
\begin{align}
\mathcal{E}(t)\leq e^{Ct}\mathcal{E}(0)\ \mbox{for all}\ t\in [0,T]. \label{statement of est on H}   \end{align}
\end{thm}
Once the weak-strong stability principle has been established we proceed by introducing the following \textit{relativistic renormalized modulated energy}. Let $(\mathbf{X}_{N}(t),\mathbf{P}_{N}(t))$ be a solution to \eqref{Relativistic N body dynamics} and let $(\rho,p)$ be a solution to \eqref{Relativistic Euler intro}. The relativistic renormalized modulated energy is denoted by $\mathcal{E}_{N}(t)$ and is defined by 
\begin{align}
\mathcal{E}_{N}(t)&=\underset{\coloneqq \mathcal{K}_{N}(t)}{\underbrace{\frac{2}{N}\sum_{i=1}^{N}    \eta(P_{i}(t)\vert p(t,X_{i}(t)))}}+\underset{\coloneqq \mathcal{V}_{N}(t)}{\underbrace{\int_{\Delta^{c}}V(x-y)\left(\rho_{N}(t,\cdot)-\rho(t,\cdot)\right)^{\otimes 2}(\dd x\dd y)}}+\mathcal{C}_{N}. \label{Renormalized relativistic modulated energy} \end{align}
As before, $\mathcal{C}_{N}$ is the constant defined in Proposition \ref{non negativity}. The kinetic part $\mathcal{K}_{N}(t)$ is non-negative due to the strict (though non-uniform) convexity of $\mathscr{K}$ (see Lemma \ref{coercivity positivity of eta}). The term $\mathcal{V}_{N}(t)+\mathcal{C}_{N}$ is non-negative due to Theorem \ref{non negativity}. The following theorem establishes the mean-field limit leading from \eqref{Relativistic N body dynamics} to \eqref{Relativistic Euler intro}, and constitutes the main result of this work.  
\begin{thm}
Suppose that: 
\begin{itemize}
    \item $(\mathbf{X}_{N}^{0},\mathbf{P}_{N}^{0})\in \Delta_{N}^{c}\times \mathbb{R}^{dN}$  and $(\mathbf{X}_{N}(t),\mathbf{P}_{N}(t))\in C^{1}([0,T];\mathbb{T}^{dN}\times \mathbb{R}^{dN})$ is a solution to \eqref{Relativistic N body dynamics}. 
    \item $(\rho^{0},p^{0})\in W^{1,a}(\mathbb{T}^{d})\times W^{2,a}(\mathbb{T}^{d})$ for some $a>d$ and $(\rho,p)\in C([0,T];W^{1,a}(\mathbb{T}^{d}))\times C([0,T];W^{2,a}(\mathbb{T}^{d}))$ is the unique solution to \eqref{Relativistic Euler intro} (as guaranteed via Theorem \ref{short time well posedness}).
    \item $\mathcal{E}_{N}(t)$ is given by \eqref{Renormalized relativistic modulated energy}.
\end{itemize}
Then, it holds that 
\begin{align}
\frac{\dd}{\dd t}\mathcal{E}_{N}(t)&= \int_{\Delta^{c}}(v(p(x))-v(p(y)))\cdot \nabla V(x-y)(\rho_{N}-\rho)^{\otimes 2}(\dd x\dd y) \notag\\
&-\frac{2}{N}\sum_{i=1}^{N}(v(P_{i})-v(p(X_{i})))((P_{i}-p(X_{i}))D_{x}(v(p))(X_{i})) \notag\\
&+\frac{2}{N}\sum_{i=1}^{N}(\nabla^{2}\mathscr{K}(p(X_{i}))(P_{i}-p(X_{i}))-(\nabla \mathscr{K}(P_{i})-\nabla \mathscr{K}(p(X_{i}))))\cdot \nabla V\ast \rho(X_{i}). \label{time der of renormalized relativistic modulated energy intro}    
\end{align}
Consequently, there is some $C=C(d,\left\Vert \rho\right\Vert_{L^{\infty}L^{a}}, \left\Vert D_{x}p\right\Vert_{\infty} )$ such that it holds that  
\begin{align}
\underset{t\in [0,T]}{\sup}\mathcal{E}_{N}(t)\leq Ce^{CT}(\mathcal{E}_{N}(0)+To_{N}(1)) \label{Gronwall est statement}     
\end{align}
and it holds that 
\begin{align}
\underset{t\in [0,T]}{\sup} W_{1}(\rho_{N}(t,\cdot ),\rho(t,\cdot))\underset{N\rightarrow \infty}{\rightarrow} 0\ \mbox{and}\ \underset{t\in [0,T]}{\sup}\left\Vert (q_{N}-\rho p)(t,\cdot)\right\Vert_{W^{-1,\infty}}\underset{N\rightarrow \infty}{\rightarrow}0. \label{convergence as N     infty statement}    
\end{align} 
\label{main thm}
\end{thm}
\subsubsection{Main difficulties in comparison to the non-relativistic settings.}\label{subsec novelty} The key difference  between the proof of Theorem \ref{main thm} and the proof of the non-relativistic mean-field limit in \cite{duerinckx2020mean} lies in the fact that the kinetic part of the relativistic renormalized modulated energy $\mathcal{K}_{N}(t)$ is given by means of $\mathscr{K}$. There is crucial difference between the case of quadratic $\mathscr{K}$ (as in the non-relativistic settings) and $\mathscr{K}$ given by \eqref{def of rel kinetic part} (which has linear growth):  the case where $\eta(P\vert p)=\left\vert P-p\right\vert^{2}$ in particular falls into the category of  functions satisfying the following "weak-strong coercivity estimate": 
\begin{align}
\eta(P\vert p)\geq C\left\vert P-p\right\vert^{2} \label{quadratic weak coer}     
\end{align}
where $C>0$ depends only on $\left\Vert p\right\Vert_{\infty}$. However, in the case where $\eta$ takes the form \eqref{definition of eta intro} an estimate of the type \eqref{quadratic weak coer} cannot hold with a constant independent of $P$. Indeed, suppose on the contrary that there is a constant $C=C_{R}>0$ such that for all $(P,p)\in \mathbb{R}^{d}\times B_{R}(0)$ it holds that \begin{align*}
\eta(P\vert p)\geq C\left\vert P-p\right\vert^{2}.     
\end{align*}
In particular, for $p=0$ we get 
\begin{align*}
\sqrt{1+\left\vert P\right\vert^{2} }-1 \geq C \left\vert P\right\vert^{2}.      
\end{align*}
Dividing by $\left\vert P\right\vert^{2}$ yields 
\begin{align*}
 C\leq \frac{\sqrt{1+\left\vert P\right\vert^{2}}}{\left\vert P\right\vert^{2}}-\frac{1}{\left\vert P\right\vert^{2}}\underset{\left\vert P \right\vert\rightarrow \infty }{\rightarrow}  0,     
\end{align*}
which contradicts the assumption $C>0$. 
We are thus lead to reject a quadratic lower bound. Instead, for $\eta$ given by \eqref{definition of eta intro} we are able to prove the following modified estimate:  
\begin{align}
\eta(P\vert p)\geq \frac{C\left\vert P-p\right\vert^{2}}{1+\left\vert P-p\right\vert } \label{weak strong coer bound}   
\end{align}
where $C$ depends only on $\left\Vert p\right\Vert_{\infty} $. That $C$ is independent of a bound on $P$ is of crucial importance, since eventually $P_{i}$ will play the role of $P$, and pointwise bounds on $P_{i}$ usually grow very fast in $N$, which would prevent passing to the limit as $N\rightarrow \infty$. The discrepancy between the estimate \eqref{weak strong coer bound} and the estimate \eqref{quadratic weak coer} is intimately related to the fact that $\nabla^{2}\mathscr{K}(z)$ is not uniformly convex for $\mathscr{K}$ given by \eqref{def of rel kinetic part} whereas it is uniformly convex when $\mathscr{K}$ is quadratic. Even after identifying that $\eta$ satisfies the bound \eqref{weak strong coer bound} it still remains to prove that the second and and third terms in the right-hand side of \eqref{time der of renormalized relativistic modulated energy intro} are bounded by means of $\frac{1}{N}\sum_{i=1}^{N}\frac{\left\vert P_{i}-p(X_{i})\right\vert^{2} }{1+\left\vert P_{i}-p(X_{i})\right\vert^{2} }$ (and consequently by means of $\frac{1}{N}\sum_{i=1}^{N}\eta(P_{i}\vert p(X_{i}))$). This requires an additional argument that is absent from the non-relativistic settings. Finally, showing that the vanishing of $\underset{t\in [0,T]}{\sup}\mathcal{E}_{N}(t)$ entails the weak convergence of $q_{N}$ to $\rho p$ also has to be modified. On a more technical level, we also need to establish the local well-posdeness of \eqref{Relativistic Euler intro}, as it appears to be absent from the litearature. 

\vspace{0.4 cm}

The paper is organized as follows. In section \ref{local well posedness of relativistic Euler} we prove the local well-posedness of \eqref{Relativistic Euler intro}. The existence part is straightforward. More important is uniqueness, which is a consequence of the weak-strong stability stated in Theorem \ref{weak strong stability intro}, and provides a good preparation towards the mean-field limit proved in section \ref{Mean field limit section}. The estimate \eqref{weak strong coer bound}, which represents part of the novelty of this work, is included in Lemma \ref{coercivity positivity of eta}. Combined with Lemma \ref{elementary ine lemma} and the calculation of the time variation of $\mathcal{E}(t)$ it enables us to close the Gr\"onwall estimate for $\mathcal{E}(t)$. In section \ref{Mean field limit section} we renormalize the argument presented in section \ref{local well posedness of relativistic Euler} in order to prove the mean-field limit stated in Theorem \ref{main thm}.        

\section{The relativistic Euler-Poisson equation: local existence and stability}\label{local well posedness of relativistic Euler}
We prove the local well-posedness of \eqref{Relativistic Euler intro}. The literature on the well-posedness of multi-dimensional relativistic equations of Euler-Poisson type appears to be limited, although the $1D$ case has been addressed in \cite{rozanova2026sufficient}, where the authors identify conditions on the initial data for which even global well-posedness is ensured. We start by proving the apriori estimates needed for local existence of classical solutions on short time intervals and then move to prove the weak-strong stability principle, which is more difficult.  
\subsection{Short time existence.} 
We concentrate on proving apriori estimates on a short time interval. It is then standard to deduce existence on a short time interval via a fixed point argument. Given that this is not the main focus of this work we will not include the details for the fixed point argument. 
\begin{thm}
Suppose that: 
\begin{itemize}
    \item $(\rho^{0},p^{0})\in W^{1,a}(\mathbb{T}^{d})\times W^{2,a}(\mathbb{T}^{d})$ for some $a>d$. 
    \item $V$ is the $d$-dimensional Coulomb interaction on $\mathbb{T}^{d}$.
\end{itemize}
Let $(\rho,p)$ is a classical solution to \eqref{Relativistic Euler intro} with initial data $(\rho^
{0},p^{0})$. Then, there is some $T=T(a,d,\left\Vert \rho^{0}\right\Vert_{W^{1,a}},\left\Vert p^{0}\right\Vert_{W^{2,a}})>0$ and a constant $C=C(a,d,\left\Vert \rho^{0}\right\Vert_{W^{1,a}},\left\Vert p^{0}\right\Vert_{W^{2,a}})>0$ such that 
\begin{align*}
\underset{t\in [0,T]}{\sup}(\left\Vert \rho(t,\cdot)\right\Vert_{W^{1,a}}+\left\Vert p(t,\cdot)\right\Vert_{W^{2,a}})\leq  C.      
\end{align*}
\label{short time well posedness}
\end{thm}
\begin{proof}
In what follows $C>0$ designates any constant which depends only on $a$ and $d$, and may vary between the terms. \\ 
\textbf{Estimate on $\frac{\dd}{\dd t}\left\Vert \nabla \rho(t,\cdot)\right\Vert_{a}^{a}$.} Taking the gradient in $x$ in the equation for $\rho$ we  obtain  
\begin{align*}
\partial_{t}\nabla\rho+\nabla\mathrm{div}(v(p)\rho)=0    
\end{align*}
so that 
\begin{align*}
\partial_{t}\nabla\rho+\nabla(\nabla\rho \cdot v(p)+\rho \mathrm{div}(v(p)))=0 \end{align*}
so that 
\begin{align*}
\partial_{t}\nabla\rho +\nabla^{2}\rho v(p)+D_{x}pD_{p}v(p)\nabla\rho +\mathrm{div}(v(p))\nabla\rho +\rho\nabla \mathrm{div}(v(p))=0.   
\end{align*}
Multiplying the above equation by $a\nabla\rho \left\vert \nabla\rho\right\vert^{a-2} $ we obtain 
\begin{align*}
\frac{\dd}{\dd t}\left\Vert \nabla \rho(t,\cdot)\right\Vert_{a}^{a}=&-\int_{\mathbb{T}^{d}} \nabla\left\vert \nabla \rho\right\vert^{a}\cdot v(p)\ \dd x-a\int_{\mathbb{T}^{d}}  (D_{x}pD_{p}v(p)\nabla \rho)\cdot\nabla \rho \left\vert \nabla \rho\right\vert^{a-2}\ \dd x\\
&-a\int_{\mathbb{T}^{d}} \left\vert \nabla \rho\right\vert^{a}\mathrm{div}(v(p))\ \dd x-a\int_{\mathbb{T}^{d}} \rho \left\vert \nabla \rho\right\vert^{a-2} \nabla \rho\cdot \nabla \mathrm{div}(v(p))\ \dd x\coloneqq\sum_{k=1}^{4}I_{k}.  \end{align*}
Integrating by parts we see that 
\begin{align*}
I_{1}+I_{3}&=-(a-1)\int_{\mathbb{T}^{d}} \left\vert \nabla \rho\right\vert^{a}\mathrm{div}(v(p))\ \dd x=-(a-1)\int_{\mathbb{T}^{d}} \left\vert \nabla \rho\right\vert^{a}D_{p}v(p):D_{x}p\ \dd x.  
\end{align*}
Note that $\left\Vert D_{p}v \right\Vert_{\infty}\leq 1 $. Therefore, using the Sobolev embedding $W^{1,a}(\mathbb{T}^{d})\hookrightarrow L^{\infty}(\mathbb{T}^{d})$ for $a>d$  and that $c\left\Vert p\right\Vert_{W^{2,a}} \leq \left\Vert \Delta p\right\Vert\leq C\left\Vert p\right\Vert_{W^{2,a}}  $ we get   
\begin{align}
I_{1}+I_{3}&\leq C\left\Vert D_{x}p(t,\cdot)\right\Vert_{\infty} \left\Vert \nabla \rho(t,\cdot)\right\Vert_{a}^{a} \notag\\
&\leq C\left\Vert p(t,\cdot)\right\Vert_{W^{2,a}}\left\Vert \nabla \rho(t,\cdot)\right\Vert_{a}^{a}  \leq C\left\Vert \Delta p(t,\cdot)\right\Vert_{a}\left\Vert \nabla \rho(t,\cdot)\right\Vert_{a}^{a}  . \label{I1I3 est}     
\end{align}
From the same considerations we estimate  
\begin{align}
I_{2}&\leq \left\Vert D_{p}v\right\Vert_{\infty} \left\Vert D_{x}p(t,\cdot)\right\Vert_{\infty} \int_{\mathbb{T}^{d}} \left\vert \nabla \rho\right\vert^{a}\ \dd x \notag\\
&\leq \left\Vert D_{x}p(t,\cdot)\right\Vert_{\infty}\left\Vert \nabla \rho(t,\cdot)\right\Vert_{a}^{a}\leq C\left\Vert \Delta p(t,\cdot)\right\Vert_{a}\left\Vert \nabla \rho(t,\cdot)\right\Vert_{a}^{a}.  
\label{I2 est}      
\end{align}
To estimate $I_{4}$ we observe that
\begin{align}
I_{4}&=-a\int_{\mathbb{T}^{d}} \rho \left\vert \nabla \rho\right\vert^{a-2}\nabla \rho\cdot \nabla (D_{p}v:D_{x}p)\ \dd x \notag\\
&\leq Ca\left\Vert \rho(t,\cdot)\right\Vert_{\infty} \int_{\mathbb{T}^{d}} \left\vert \nabla \rho\right\vert^{a-1}(\left\vert D_{x}p\right\vert^{2}+\left\vert D_{x}^{2}p\right\vert)\ \dd x \notag\\
&\leq C \left\Vert \rho(t,\cdot)\right\Vert_{\infty}\left(\left\Vert \nabla \rho(t,\cdot)\right\Vert_{a}^{a}+\left\Vert D_{x}p(t,\cdot)\right\Vert_{2a}^{2a}  +\left\Vert \Delta p(t,\cdot)\right\Vert_{a}^{a}\right) \notag\\
&\leq C\left\Vert \nabla \rho(t,\cdot)\right\Vert_{a}\left(\left\Vert \nabla \rho(t,\cdot)\right\Vert_{a}^{a} +\left\Vert \Delta p(t,\cdot)\right\Vert_{a}^{a}+\left\Vert \Delta p(t,\cdot)\right\Vert_{a}^{2a} \right) \notag\\
& \leq C(\left\Vert \nabla \rho(t,\cdot)\right\Vert_{a}^{a+1}+\left\Vert \nabla \rho(t,\cdot)\right\Vert_{a}\left\Vert \Delta p(t,\cdot)\right\Vert_{a}^{a}+\left\Vert \nabla \rho(t,\cdot)\right\Vert_{a}\left\Vert \Delta p(t,\cdot)\right\Vert_{a}^{2a})   ,  \label{I4 est}    
\end{align}
where in the third inequality we used the embedding $W^{1,a}(\mathbb{T}^{d})\hookrightarrow L^{\infty}(\mathbb{T}^{d})$. 
Gathering \eqref{I1I3 est}-\eqref{I4 est} we arrive at the estimate 
\begin{align}
\frac{\dd}{\dd t}\left\Vert \nabla \rho(t,\cdot)\right\Vert_{a}^{a}\leq& C\Big(\left\Vert \Delta p(t,\cdot)\right\Vert_{a}\left\Vert \nabla \rho(t,\cdot)\right\Vert_{a}^{a}+\left\Vert \nabla \rho(t,\cdot)\right\Vert_{a}\left\Vert \Delta p(t,\cdot)\right\Vert_{a}^{a} \notag\\
&+\left\Vert \nabla \rho(t,\cdot)\right\Vert_{a}^{a+1}+\left\Vert \nabla \rho(t,\cdot)\right\Vert_{a}\left\Vert \Delta p(t,\cdot)\right\Vert_{a}^{2a}\Big) \notag\\
\leq& C(\left\Vert \Delta p(t,\cdot)\right\Vert_{a}^{\frac{2a^{2}}{a-1}} +\left\Vert \nabla \rho(t,\cdot)\right\Vert_{a}^{a}+\left\Vert \nabla \rho(t,\cdot)\right\Vert_{a}^{a+1}+\left\Vert \Delta p(t,\cdot)\right\Vert_{a}^{a}+\left\Vert \Delta p(t,\cdot)\right\Vert_{a}^{a+1}).  
\label{ine for nablarho}         
\end{align}
\textbf{Estimate on $\frac{\dd}{\dd t}\left\Vert \Delta p(t,\cdot)\right\Vert_{a}^{a}$.} Taking the Laplacian in the equation for $p$ we find that 
\begin{align*}
\frac{\dd}{\dd t}\Delta p+\Delta(v(p)D_{x}p)=-\Delta\nabla V\ast \rho    
\end{align*}
and so using that $-\Delta V=\delta_{0}-1$ we find 
\begin{align*}
\frac{\dd}{\dd t}\Delta p+\Delta(v(p)D_{x}p)=\nabla \rho.     
\end{align*}
Expanding, we obtain 
\begin{align*}
\frac{\dd}{\dd t}\Delta p+\Delta(v(p))D_{x}p+2\sum_{k=1}^{d}\partial_{x_{k}}(v(p)) D_{x}\partial_{x_{k}}p+v(p)D_{x}\Delta p=\nabla \rho .     
\end{align*}
Multiplying by $a\Delta p\left\vert \Delta p\right\vert^{a-2}$, it follows that 
\begin{align*}
\frac{\dd}{\dd t}\left\Vert \Delta p(t,\cdot)\right\Vert_{a}^{a}=&-a\int_{\mathbb{T}^{d}} (\Delta (v(p))D_{x}p)\cdot \Delta p\left\vert \Delta p\right\vert^
{a-2}\ \dd x-\int_{\mathbb{T}^{d}} v(p)\cdot \nabla \left\vert \Delta p\right\vert ^{a}\ \dd x\\
&-2a\sum_{k=1}^{d}\int_{\mathbb{T}^{d}} (\partial_{x_{k}}(v(p))D_{x}\partial_{x_{k}}p)\cdot \Delta p\left\vert \Delta p\right\vert^{a-2}\ \dd x\\
&+a\int_{\mathbb{T}^{d}}\nabla \rho \cdot \Delta p\left\vert \Delta p\right\vert^{a-2}\ \dd x \coloneqq \sum_{k=1}^{4}J_{k}.    
\end{align*}
To estimate $J_{1}$, we use the formula $\Delta(v(p))=D_{p}v(p)\Delta p+D^{2}_{p}v(p):(D_{x}^{T}pD_{x}p)$ to find that  
\begin{align}
J_{1}&=-a\int_{\mathbb{T}^{d}} (D_{p}v(p)\Delta p+D_{p}^{2}v(p):(D_{x}^{T}pD_{x}p))\cdot \Delta p\left\vert \Delta p\right\vert^{a-2}\ \dd x \notag\\
&=-a\int_{\mathbb{T}^{d}} D_{p}v(p):\Delta p\otimes \Delta p\left\vert \Delta p\right\vert^{a-2}\ \dd x-a\int_{\mathbb{T}^{d}} (D_{p}^{2}v(p):D_{x}^{T}pD_{x}p) \cdot \Delta p\left\vert \Delta p\right\vert^{a-2}\ \dd x.   \label{rhs ofJ1}
\end{align}
The first term in the right-hand side of \eqref{rhs ofJ1} is bounded by 
\begin{align*}
C\left\Vert D_{p}v\right\Vert_{\infty}\left\Vert \Delta p(t,\cdot)\right\Vert_{a}^{a}\leq C\left\Vert \Delta p(t,\cdot)\right\Vert_{a}^{a} 
\end{align*}
whereas the second term in the right-hand side of \eqref{rhs ofJ1} is bounded by 
\begin{align*}
&C\left\Vert D_{p}^{2}v\right\Vert_{\infty}\left\Vert D_{x}p(t,\cdot)\right\Vert_{\infty}\int_{\mathbb{T}^{d}} \left\vert D_{x}p\right\vert\left\vert \Delta p\right\vert^{a-1}\ \dd x \\
&\leq C\left\Vert D
_{x}p(t,\cdot)\right\Vert_{\infty}\left(\left\Vert D_{x}p(t,\cdot)\right\Vert_{a}^{a}+\left\Vert \Delta p(t,\cdot)\right\Vert_{a}^{a}  \right)\\
&\leq C\left\Vert D_{x}p(t,\cdot)\right\Vert_{\infty}\left\Vert \Delta p(t,\cdot)\right\Vert_{a}^{a}  \leq C\left\Vert \Delta p(t,\cdot)\right\Vert_{a}^{a+1}.          
\end{align*}
Therefore we obtain
\begin{align}
J_{1}\leq C(\left\Vert \Delta p(t,\cdot)\right\Vert_{a}^{a}+\left\Vert \Delta p(t,\cdot)\right\Vert_{a}^{a+1}). \label{est J1}   
\end{align}
To estimate $J_{2}$, we integrate by parts to find 
\begin{align}
J_{2}&=\int_{\mathbb{T}^{d}} \mathrm{div}(v(p))\left\vert \Delta p\right\vert^{a}\ \dd x=\int_{\mathbb{T}^{d}} D_{p}v(p):D_{x}p\left\vert \Delta p\right\vert^{a}\ \dd x \notag\\
&\leq C\left\Vert D_{x}p(t,\cdot)\right\Vert_{\infty} \left\Vert \Delta p(t,\cdot)\right\Vert_{a}^{a}   \leq C\left\Vert \Delta p(t,\cdot)\right\Vert_{a}^{a+1}.
\label{J2 est}     
\end{align}
To estimate $J_{3}$ we observe that 
\begin{align}
J_{3}\leq C\left\Vert D_{x}p(t,\cdot)\right\Vert_{\infty} \int_{\mathbb{T}^{d}}\left\vert D_{x}^{2}p\right\vert \left\vert \Delta p\right\vert^{a-1}\ \dd x\leq C\left\Vert D_{x}p(t,\cdot)\right\Vert_{\infty}\left\Vert \Delta p(t,\cdot)\right\Vert_{a}^{a}\leq C\left\Vert \Delta p(t,\cdot)\right\Vert_{a}^{a+1}.         
\label{J3 est}     
\end{align}
As for $J_{4}$, we have 
\begin{align}
J_{4}\leq a\int_{\mathbb{T}^{d}}\left\vert \nabla \rho\right\vert\left\vert \Delta p\right\vert^{a-1}\ \dd x\leq  C(\left\Vert \nabla \rho(t,\cdot)\right\Vert_{a}^{a}+\left\Vert \Delta p(t,\cdot)\right\Vert_{a}^{a}).  \label{J4 est}     
\end{align}
So, to conclude, gathering \eqref{est J1}-\eqref{J4 est} we obtain 
\begin{align}
\frac{\dd}{\dd t}\left\Vert \Delta p(t,\cdot)\right\Vert_{a}^{a}\leq C(\left\Vert \nabla \rho(t,\cdot)\right\Vert_{a}^{a} +\left\Vert \Delta p(t,\cdot)\right\Vert_{a}^{a}+\left\Vert \Delta p(t,\cdot)\right\Vert_{a}^{a+1}  ).  \label{ine of Delta p}    
\end{align}
Set $S(t)=1+\left\Vert \nabla \rho(t,\cdot)\right\Vert_{a}^{a}+\left\Vert \Delta p(t,\cdot)\right\Vert_{a}^{a}$. In view of  \eqref{ine for nablarho} and \eqref{ine of Delta p}  we find that 
\begin{align*}
\frac{\dd}{\dd t}S(t)\leq CS^{\mathfrak{a}}(t) \ \mbox{for some}\ \mathfrak{a}>1.      
\end{align*}
Solving the above differential inequality we infer that $\underset{t\in [0,T]}{\sup}S(t)\leq C$ for some $T>0$ and $C>0$ both depending only on $a,d,\left\Vert \rho^{0}\right\Vert_{W^{1,a}},\left\Vert p^{0}\right\Vert_{W^{2,a}}$, which establishes the desired apriori estimate. 
\end{proof}
\subsection{Weak-strong stability.} Prior to proving the weak-strong stability principle (as stated in Theorem \ref{weak strong stability intro}) we will need several key preliminary lemmas.  We start with the following elementary inequality, which will be repeatedly used in the sequel.  
\begin{lem}\label{elementary ine about min}
For all $r\geq 0$ it holds that 
\begin{align*}
\frac{r^{2}}{1+r}\leq \min\{r,r^{2}\}\leq \frac{2r^{2}}{1+r}.     
\end{align*}
\end{lem}
\begin{proof}
Start with the upper bound. If $0\leq r\leq 1$ then 
\begin{align*}
\min\{r,r^{2}\}=r^{2}\leq \frac{2r^{2}}{1+r},    
\end{align*}
whereas if $r>1$ then 
\begin{align*}
\min\{r,r^{2}\}=r=\frac{2r^{2}}{r+r}\leq \frac{2r^{2}}{1+r}.     
\end{align*}
As for the lower bound, if $0\leq r\leq 1$ we have 
\begin{align*}
\min\{r,r^{2}\}=r^{2}\geq \frac{r^{2}}{1+r},    
\end{align*}
whereas if $r>1$ then 
\begin{align*}
\min\{r,r^{2}\}=r=\frac{2r^{2}}{2r}\geq \frac{2r^{2}}{2+2r}=\frac{r^{2}}{1+r}.      
\end{align*}
\end{proof}
The following Lemma is a key component in  closing the Gr\"onwall estimate for $\mathcal{E}(t)$ and  $\mathcal{E}_{N}(t)$. Note that in particular it explains why $\eta(p_{1}\vert p_{2})$ is non-negative (and so the same is true for $\mathcal{E}(t)$).   
\begin{lem}\label{coercivity positivity of eta}
There is a constant $C_{R}>0$ such that for any $(p_{1},p_{2})\in \mathbb{R}^{d}\times B_{R}(0)$ it holds that 
 \begin{align*}
 \eta(p_{1}\vert p_{2})\geq \frac{C_{R}\left\vert p_{1}-p_{2}\right\vert^{2}}{1+\left\vert p_{1}-p_{2}\right\vert}.     
 \end{align*}
In particular, $\eta(p_{1}\vert p_{2})\geq 0$ for all $p_{1},p_{2}\in \mathbb{R}^{d}$. 
\end{lem}
\begin{proof}
In view of Lemma \ref{elementary ine about min} it suffices to prove that there is some $C_{R}>0$ such that for all $(p_{1},p_{2})\in \mathbb{R}^{d}\times B_{R}(0)$ it holds that 
\begin{align}
\eta(p_{1}\vert p_{2})\geq C_{R}\min\{\left\vert p_{1}-p_{2}\right\vert ,\left\vert p_{1}-p_{2}\right\vert^{2} \}. \label{eta requested inequality}   \end{align}
To achieve this inequality we proceed via the following steps. \\
\textbf{Calculation of the eigenvalues of $\nabla^{2}\mathscr{K}(z)$.} We claim that for any $z\neq 0$ the Hessian matrix $\nabla^{2}\mathscr{K}(z)$ has exactly $2$ positive eigenvalues (up to multiplicity) given by 
\begin{align}
\lambda_{\max}(z)=\frac{1}{\sqrt{1+\left\vert z\right\vert^{2}}} \ \mbox{and}\ \lambda_{\min}(z)=\frac{1}{(1+\left\vert z\right\vert^{2})^{\frac{3}{2}}}. \label{formula for eigenvalues}    
\end{align}
First, by direct calculation we find that the Hessian of $\mathscr{K}$ is given by 
\begin{align*}
\nabla^{2}\mathscr{K}(z)=\frac{1}{\sqrt{1+\left\vert z\right\vert^{2} }}I-\frac{z\otimes z}{(1+\left\vert z\right\vert^{2})^{\frac{3}{2}}}.     
\end{align*}
If $v\perp z$ then we have 
\begin{align*}
(z\otimes z)v=(z\cdot v)z=0.   
\end{align*}
Therefore it follows that 
\begin{align*}
\nabla^{2}\mathscr{K}(z)v=\frac{1}{\sqrt{1+\left\vert z\right\vert^{2}}}Iv=\frac{1}{\sqrt{1+\left\vert z\right\vert^{2}}}v.
\end{align*}
It follows that $\frac{1}{\sqrt{1+\left\vert z\right\vert^{2}}}$ is an eigenvalue with eigenspace $z^{\perp}$. Therefore, since $\nabla^{2}\mathscr{K}(z)$ is symmetric the multiplicity of $\frac{1}{\sqrt{1+\left\vert z\right\vert^{2}}}$ is $\dim(z^{\perp})=d-1$. On the other hand, using the identity $$(z\otimes z)z=\left\vert z\right\vert^{2}z$$ we get  
\begin{align*}
\nabla^{2}\mathscr{K}(z)z=\frac{1}{\sqrt{1+\left\vert z\right\vert^{2} }}z-\frac{\left\vert z\right\vert^{2} }{(1+\left\vert z\right\vert^{2} )^{\frac{3}{2}}}z=\left(\frac{1}{\sqrt{1+\left\vert z\right\vert^{2}}}-\frac{\left\vert z\right\vert^{2}}{(1+\left\vert z\right\vert^{2} )^{\frac{3}{2}}}\right)z=\frac{1}{(1+\left\vert z\right\vert^{2})^{\frac{3}{2}}}z.    
\end{align*}
It follows that $\frac{1}{(1+\left\vert z\right\vert^{2})^{\frac{3}{2}}}$ is an eigenvalue with eigenspace $\mathrm{span}\{z\}$. Since $z\neq 0$ both of the eigenvalues we found are distinct. Moreover $\dim(z^{\perp})=d-1$ while $\dim(\mathrm{span}\{z\}{})=1$ and therefore these are the only eigenvalues. 
To conclude, for any $z\neq 0$ the Hessian  $\nabla^{2}\mathscr{K}(z)$ has exactly two distinct eigenvalues given by \ref{formula for eigenvalues}.
We will now separate between two regions: Fix $M=M_{R}$ to be chosen later.  We will first consider the case where $\left\vert p_{1}-p_{2}\right\vert\leq M $ and then the case where $\left\vert p_{1}-p
_{2}\right\vert\geq M$. We may assume with no loss of generality that $p_{1}\neq p_{2}$. \\ 
\textbf{The case where $\left\vert p_{1}-p_{2}\right\vert\leq M$}. By the theorem of Taylor-Hadamard we have 
\begin{align*}
\eta(p_{1}\vert p_{2})&=\mathscr{K}(p_{1})-\mathscr{K}(p_{2})-\nabla \mathscr{K}(p_{2})\cdot (p_{1}-p_{2})\\
&=\int_{0}^{1}(\nabla \mathscr{K}(p_{2}+s(p_{1}-p_{2}))-\nabla \mathscr{K}(p_{2}))\cdot (p_{1}-p_{2})\ \dd s\\
&= \int_{0}^{1} s(p_{1}-p_{2})^{T}\nabla^{2}\mathscr{K}(\xi_{s})(p_{1}-p_{2})\ \dd s
\end{align*}
where in the last equation we applied the intermediate value theorem with an intermediate point $\xi_{s}=(1-\theta )p_{2}+\theta(p_{2}+s(p_{1}-p_{2}))=(1-\theta s)p_{2}+\theta sp_{1}$ for some $\theta \in [0,1]$. Consequently, we deduce that \footnote{Recall that if $A$ is a matrix with minimal eigenvalue $\lambda_{\min}$ then $\lambda_{\min}=\underset{\left\vert v\right\vert=1 }{\min}v^{T}Av$. }  
\begin{align}
\eta(p_{1}\vert p_{2})\geq \left\vert p_{1}-p_{2}\right\vert^{2} \int_{0}^{1}s\min_{\left\vert v\right\vert=1 }v^{T}\nabla^{2}\mathscr{K}(\xi_{s})v\ \dd s
=\left\vert p_{1}-p_{2}\right\vert^{2} \int_{0}^{1}s\lambda_{\min}(\xi_{s})\ \dd s. \label{eta(p1p2) est in term integral}   
\end{align}
In view of \eqref{formula for eigenvalues} we have that 
\begin{align*}
\int_{0}^{1}s\lambda_{\min}(\xi_{s})\ \dd s&=\int_{0}^{1}\frac{s}{(1+\left\vert \xi_{s}\right\vert^{2})^{\frac{3}{2}}}\ \dd s\\
&=\int_{0}^{1}\frac{s}{(1+\left\vert (1-\theta s)p_{2}+\theta s p_{1}\right\vert^{2})^{\frac{3}{2}}}\ \dd s.    \end{align*}
In addition, note that  
\begin{align*}
\left\vert (1-\theta s)p_{2}+\theta sp_{1}\right\vert^{2}&\leq 2(1-\theta s)^{2}\left\vert p_{2}\right\vert^{2}+2\theta^{2}s^{2} \left\vert p_{1}\right\vert^{2}\\
&\leq 2\left\vert p_{2}\right\vert^{2}+4(\left\vert p_{1}-p_{2}\right\vert^{2}+\left\vert p_{2}\right\vert^{2} 
)\leq 6R^{2}+4M^{2}.     \end{align*}
Hence, we infer  
\begin{align}
\int_{0}^{1} s\lambda_{\min}(\xi_{s})\ \dd s\geq \frac{1}{(1+6R^{2}+4M^{2})^{\frac{3}{2}}}. \label{lower bound on integral}    
\end{align}
Substituting \eqref{lower bound on integral} inside \eqref{eta(p1p2) est in term integral} we conclude that 
\begin{align}
\eta(p_{1}\vert p_{2})\geq \frac{1}{(1+6R^{2}+4M^{2})^{\frac{3}{2}}}\left\vert p_{1}-p_{2}\right\vert^{2}. \label{quadratic est on eta}   
\end{align}
\textbf{The case where $\left\vert p_{1}-p_{2}\right\vert\geq M $.} First, we claim to have the inequality 
\begin{align}
\mathscr{K}(p_{1})-\nabla \mathscr{K}(p_{2})\cdot p_{1}\geq \left\vert p_{1}\right\vert\left(1-\frac{R}{\sqrt{1+R^{2}}}\right).\label{inequality for K}     
\end{align}
Indeed 
\begin{align*}
\mathscr{K}(p_{1})-\nabla \mathscr{K}(p_{2})\cdot p_{1}&=\sqrt{1+\left\vert p_{1}\right\vert^{2}}-\frac{p_{2}\cdot p_{1}}{\sqrt{1+\left\vert p_{2}\right\vert^{2}  }}\\
&\geq \left\vert p_{1}\right\vert-\frac{\left\vert p_{1}\right\vert\left\vert p_{2}\right\vert  }{\sqrt{1+\left\vert p_{2}\right\vert^{2}}}=\left\vert p_{1}\right\vert\left(1-\frac{\left\vert p_{2}\right\vert}{\sqrt{1+\left\vert p_{2}\right\vert^{2} }}\right).       
\end{align*}
Since $r\mapsto \frac{r}{\sqrt{1+r^{2}}}$ is increasing and $\left\vert p_{2}\right\vert\leq R$ it follows that 
\begin{align*}
1-\frac{\left\vert p_{2}\right\vert}{\sqrt{1+\left\vert p_{2}\right\vert^{2}}}\geq 1-\frac{R}{\sqrt{1+R^{2}}}    
\end{align*}
which establishes \eqref{inequality for K}. 
Owing to \eqref{inequality for K} and noticing the identity 
$$\nabla \mathscr{K}(p_{2})\cdot p_{2}-\mathscr{K}(p_{2})=\frac{\left\vert p_{2}\right\vert^{2} }{\sqrt{1+\left\vert p_{2}\right\vert^{2} }}-\sqrt{1+\left\vert p_{2} \right\vert^{2} }=-\frac{1}{\sqrt{1+\left\vert p_{2}\right\vert^{2}}}$$we infer that 
\begin{align*}
\eta(p_{1}\vert p_{2})&=
\mathscr{K}(p_{1})-\nabla \mathscr{K}(p_{2})\cdot p_{1}+\nabla \mathscr{K}(p_{2})\cdot p_{2}-\mathscr{K}(p_{2})\\
&\geq \left\vert p_{1}\right\vert\left(1-\frac{R}{\sqrt{1+R^{2}}}\right)-\frac{1}{\sqrt{1+\left\vert p_{2}\right\vert^{2}}}\\
&\geq \left\vert \left\vert p_{2}-p_{1}\right\vert-\left\vert p_{2}\right\vert\right\vert\left(1-\frac{R}{\sqrt{1+R^{2}}}\right)-\frac{1}{\sqrt{1+\left\vert p_{2}\right\vert^{2}}}. \end{align*}
Since $\left\vert p_{2}\right\vert\leq R $ and $\left\vert p_{1}-p_{2}\right\vert\geq M$, we have $\left\vert \left\vert p_{1}-p_{2}\right\vert -\left\vert p_{2}\right\vert \right\vert=\left\vert p_{1}-p_{2}\right\vert  -\left\vert p_{2}\right\vert $ provided $M\geq R$.  So, as long as $M\geq R$ we get 
\begin{align*}
\eta(p_{1}\vert p_{2})&\geq \left\vert p_{1}-p_{2}\right\vert\left(1-\frac{R}{\sqrt{1+R^{2}}}\right)-\left\vert p_{2}\right\vert\left(1-\frac{R}{\sqrt{1+R^{2}}}\right)-\frac{1}{\sqrt{1+\left\vert p_{2}\right\vert^{2}}}\\
&\geq \left\vert p_{1}-p_{2}\right\vert\left(1-\frac{R}{\sqrt{1+R^{2}}}\right)-R\left(1-\frac{R}{\sqrt{1+R^{2}}}\right)-1. 
\end{align*}
Setting 
\begin{align*}
\alpha_{R}\coloneqq 1-\frac{R}{\sqrt{1+R^{2}}}\ \mbox{and}\ \beta_{R}\coloneqq R\left(1-\frac{R}{\sqrt{1+R^{2}}}\right)+1    
\end{align*}
the latter inequality writes 
\begin{align*}
\eta(p_{1}\vert p_{2})\geq \alpha_{R}\left\vert p_{1}-p_{2}\right\vert-\beta_{R}.      
\end{align*}
Now set $M\coloneqq \max\{R,\frac{2\beta_{R}}{\alpha_{R}}\}$ and note that 
\begin{align*}
\alpha_{R}\left\vert p_{1}-p_{2}\right\vert-\beta_{R}\geq \frac{\alpha_{R}}{2}\left\vert p_{1}-p_{2}\right\vert\iff \left\vert p_{1}-p_{2}\right\vert\geq \frac{2\beta_{R}}{\alpha_{R}}.        
\end{align*}
So choosing  $M=\max\{R,\frac{2\beta_{R}}{\alpha_{R}}\}$ we conclude that 
\begin{align}
\eta(p_{1}\vert p_{2})\geq \frac{\alpha_{R}}{2}\left\vert p_{1}-p_{2}\right\vert. \label{linear lower bound}     
\end{align}
Combining \eqref{quadratic est on eta} with \eqref{linear lower bound} yields \eqref{eta requested inequality}. 
\end{proof}
We will also need the following inequalities, which reflect an additional particularity of the relativistic settings. 
\begin{lem}\label{elementary ine lemma}
Let $v$ and $\mathscr{K}$ be given by \eqref{def of v} and \eqref{def of rel kinetic part}, respectively. Then, there is some universal constant $C>0$ such that for all $p_{1},p_{2}\in \mathbb{R}^{d}$ it holds that: 
\begin{align}
\left\vert v(p_{1})-v(p_{2})\right\vert \left\vert p_{1}-p_{2}\right\vert\leq \frac{C\left\vert p_{1}-p_{2}\right\vert^{2}}{1+\left\vert p_{1}-p_{2}\right\vert} \label{first preliminary ine}    
\end{align}
and 
\begin{align}
\left\vert \nabla \mathscr{K}(p_{1})-\nabla \mathscr{K}(p_{2})-\nabla^{2}\mathscr{K}(p_{2})(p_{1}-p_{2})\right\vert \leq \frac{C\left\vert p_{1}-p_{2}\right\vert^{2}}{1+\left\vert p_{1}-p_{2}\right\vert}. \label{second preliminary ine}     
\end{align}
\end{lem}
\begin{proof}
We start with \eqref{first preliminary ine}. Since $v$ is Lipschitz with Lipschitz constant $1$ we have  
\begin{align*}
\left\vert v(p_{1})-v(p_{2})\right\vert\left\vert  p_{1}-p_{2}\right\vert\leq \left\vert p_{1}-p_{2}\right\vert^{2}.      
\end{align*}
In addition, since $\left\vert v(p)\right\vert\leq 1 $ it follows that 
\begin{align*}
\left\vert v(p_{1})-v(p_{2})\right\vert \left\vert p_{1}-p_{2}\right\vert\leq 2\left\vert p_{1}-p_{2}\right\vert.      
\end{align*}
Therefore we conclude that 
\begin{align*}
\left\vert v(p_{1})-v(p_{2})\right\vert \left\vert  p_{1}-p_{2}\right\vert\leq 2\min\{\left\vert p_{1}-p_{2}\right\vert,\left\vert p_{1}-p_{2}\right\vert^{2}\}. \end{align*}
In view of Lemma \ref{elementary ine about min} we infer that 
\begin{align*}
\left\vert v(p_{1})-v(p_{2})\right\vert \left\vert  p_{1}-p_{2}\right\vert\leq \frac{4\left\vert p_{1}-p_{2}\right\vert^{2}}{1+\left\vert p_{1}-p_{2}\right\vert}.      
\end{align*}
To prove \eqref{second preliminary ine} we first note that  $\left\Vert \nabla^{2}\mathscr{K}\right\Vert_{\infty}\leq 1$ and therefore 
\begin{align}
&\left\vert \nabla\mathscr{K}(p_{1})-\nabla\mathscr{K}(p_{2})-\nabla^{2}\mathscr{K}(p_{2})(p_{1}-p_{2})\right\vert \notag\\
&\leq \left\vert \nabla \mathscr{K}(p_{1})-\nabla \mathscr{K}(p_{2})\right\vert+\left\vert \nabla^{2} \mathscr{K}(p_{2})(p_{1}-p_{2})\right\vert\leq 2\left\vert p_{1}-p_{2}\right\vert. \label{linearbound}       
\end{align}
Moreover, by the theorem of Taylor-Hadamard we may write 

\begin{align*}
&\nabla^{2}\mathscr{K}(p_{2})(p_{1}-p_{2})-(\nabla \mathscr{K}(p_{1})-\nabla \mathscr{K}(p_{2}))\\
&= \int_{0}^{1}\left(\nabla^{2}\mathscr{K}(p_{2})-\nabla^{2}\mathscr{K}(p_{2}+s(p_{1}-p_{2}))\right)(p_{1}-p_{2})\ \dd s.   
\end{align*}    
Therefore, it follows that 
\begin{align}
&\left\vert \nabla^{2}\mathscr{K}(p_{2})(p_{1}-p_{2})-(\nabla \mathscr{K}(p_{1})-\nabla \mathscr{K}(p_{2}))\right\vert \notag\\
&\leq \int_{0}^{1} s\left\Vert D^{3}\mathscr{K}\right\Vert_{\infty}  \left\vert p_{1}-p_{2}\right\vert^{2}\ \dd s\leq C\left\vert p_{1}-p_{2}\right\vert^{2}.   \label{quadbound}  
\end{align}
Thus, collecting \eqref{linearbound}-\eqref{quadbound} we obtain 
\begin{align*}
\left\vert \nabla^{2}\mathscr{K}(p_{2})(p_{1}-p_{2})-(\nabla \mathscr{K}(p_{1})-\nabla \mathscr{K}(p_{2}))\right\vert&\leq \min\{2\left\vert p_{1}-p_{2}\right\vert,C\left\vert p_{1}-p_{2}\right\vert^{2} \}\\
&\leq \max\{2,C\}\min\{\left\vert p_{1}-p_{2}\right\vert ,\left\vert p_{1}-p_{2}\right\vert^{2}\}.    
\end{align*}
In view of Lemma \eqref{elementary ine about min} we deduce the estimate \eqref{second preliminary ine}. 
\end{proof}
We devote the remaining part of this section to the proof of Theorem \ref{weak strong stability intro}.

\vspace{0.4 cm}

\textbf{Proof of Theorem \ref{weak strong stability intro}.}
\textbf{Time variation of $\mathcal{K}$.} 
For brevity we set $v_{i}=v(p_{i})$. Integration by parts shows that 
\begin{align}
\frac{\dd}{\dd t} \mathcal{K}(t)=2\int_{\mathbb{T}^{d}} \partial_{t}\rho_{1}\eta(p_{1}\vert p_{2})\ \dd x+2\int_{\mathbb{T}^{d}} \rho_{1} \partial_{t}\eta(p_{1}\vert p_{2})\ \dd x =2\int_{\mathbb{T}^{d}} \rho_{1}(\partial_{t}+v_{1}\cdot \nabla)\eta(p_{1}\vert p_{2})\ \dd x. 
\label{first time der of K}
\end{align}
Let $\mathbf{D}_{i}$ be the linear differential operator defined by  $\mathbf{D}_{i}\coloneqq\partial_{t}+v_{i}D_{x}$. 
We compute that 
\begin{align*}
(\partial_{t}+v_{1}\cdot \nabla)\eta(p_{1}\vert p_{2})&=(\partial_{t}+v_{1}\cdot \nabla)\left[\mathscr{K}(p_{1})-\mathscr{K}(p_{2})-\nabla{\mathscr{K}}(p_{2})\cdot (p_{1}-p_{2})\right]\\
&=\nabla \mathscr{K}(p_{1})\cdot \mathbf{D}_{1}p_{1}-\nabla\mathscr{K}(p_{2})\cdot \mathbf{D}_{1}p_{2}\\
&-(\nabla^{2}\mathscr{K}(p_{2})\mathbf{D}_{1}p_{2})\cdot (p_{1}-p_{2})-\nabla\mathscr{K}(p_{2})\cdot \mathbf{D}_{1}(p_{1}-p_{2}).
\end{align*}
Observing that $\nabla \mathscr{K}=v$ we may recast $\nabla \mathscr{K}(p_{1})\cdot \mathbf{D}_{1}p_{1}$ as follows:  
\begin{align*}
&\nabla\mathscr{K}(p_{1})\cdot \mathbf{D}_{1}p_{1}\\
&=(\nabla \mathscr{K}(p_{1})-\nabla\mathscr{K}(p_{2}))\cdot \mathbf{D}_{1}p_{1}+\nabla\mathscr{K}(p_{2})\cdot \mathbf{D}_{1}p_{1}\\
&=(\nabla \mathscr{K}(p_{1})-\nabla \mathscr{K}(p_{2}))\cdot(\mathbf{D}_{1}p_{1}-\mathbf{D}_{2}p_{2})\\
&+(\nabla \mathscr{K}(p_{1})-\nabla\mathscr{K}(p_{2}))\cdot \mathbf{D}_{2}p_{2}+\nabla\mathscr{K}(p_{2})\cdot\mathbf{D}_{1}p_{1}\\
&=(v_{1}-v_{2})\cdot \nabla V\ast(\rho_{2}-\rho_{1})-(\nabla \mathscr{K}(p_{1})-\nabla\mathscr{K}(p_{2}))\cdot\nabla V\ast \rho_{2}+\nabla\mathscr{K}(p_{2})\cdot \mathbf{D}_{1}p_{1}\\
&=(v_{1}-v_{2})\cdot \nabla V\ast (\rho_{2}-\rho_{1})-(\nabla \mathscr{K}(p_{1})-\nabla \mathscr{K}(p_{2}))\cdot\nabla V\ast \rho_{2}+\nabla\mathscr{K}(p_{2})\cdot\mathbf{D}_{1}p_{1}.
\end{align*}
Therefore we find that 
\begin{align*}
&(\partial_{t}+v_{1}\cdot \nabla)\eta(p_{1}\vert p_{2})\\
&=(v_{1}-v_{2})\cdot\nabla V\ast (\rho_{2}-\rho_{1})-(\nabla\mathscr{K}(p_{1})-\nabla\mathscr{K}(p_{2}))\cdot\nabla V\ast \rho_{2}+\nabla\mathscr{K}(p_{2})\cdot\mathbf{D}_{1}p_{1}\\
&-\nabla\mathscr{K}(p_{2})\cdot \mathbf{D}_{1}p_{2}-(\nabla^{2}\mathscr{K}(p_{2})\mathbf{D}_{1}p_{2})\cdot(p_{1}-p_{2})-\nabla\mathscr{K}(p_{2})\cdot \mathbf{D}_{1}(p_{1}-p_{2})\\
&=(v_{1}-v_{2})\cdot \nabla V\ast (\rho_{2}-\rho_{1})-(\nabla\mathscr{K}(p_{1})-\nabla\mathscr{K}(p_{2}))\cdot \nabla V\ast \rho_{2}\\
&-(\nabla^{2}\mathscr{K}(p_{2})(\mathbf{D}_{1}p_{2}-\mathbf{D}_{2}p_{2}))\cdot (p_{1}-p_{2})-(\nabla^{2}\mathscr{K}(p_{2})\mathbf{D}_{2}p_{2})\cdot (p_{1}-p_{2})\\
&+\nabla\mathscr{K}(p_{2})\cdot\mathbf{D}_{1}p_{1}-\nabla\mathscr{K}(p_{2})\cdot \mathbf{D}_{1}p_{2}-\nabla\mathscr{K}(p_{2})\cdot \mathbf{D}_{1}(p_{1}-p_{2})\\
&=(v_{1}-v_{2})\cdot\nabla V\ast (\rho_{2}-\rho_{1})+\nabla V\ast \rho_{2}\cdot \left(\nabla^{2}\mathscr{K}(p_{2})(p_{1}-p_{2})-(\nabla\mathscr{K}(p_{1})-\nabla\mathscr{K}(p_{2}))\right)\\
&-
(\nabla^{2}\mathscr{K}(p_{2})(\mathbf{D}_{1}-\mathbf{D}_{2})p_{2})\cdot (p_{1}-p_{2})+\nabla \mathscr{K}(p_{2})\cdot \mathbf{D}_{1}(p_{1}-p_{2})-\nabla \mathscr{K}(p_{2})\cdot \mathbf{D}_{1}(p_{1}-p_{2})\\
&=(v_{1}-v_{2})\cdot \nabla V\ast (\rho_{2}-\rho_{1})+\nabla V\ast \rho_{2}\cdot (\nabla^{2}\mathscr{K}(p_{2})(p_{1}-p_{2})-(\nabla \mathscr{K}(p_{1})-\nabla \mathscr{K}(p_{2})))\\
&-(\nabla^{2} \mathscr{K}(p_{2})(\mathbf{D}_{1}-\mathbf{D}_{2})p_{2})\cdot (p_{1}-p_{2}). 
\end{align*}
We compute that 
\begin{align*}
(\mathbf{D}_{1}-\mathbf{D}_{2})p_{2}=(\partial_{t}p_{2}+v_{1}D_{x}p_{2})-(\partial_{t}p_{2}+v_{2}D_{x}p_{2})=(v_{1}-v_{2})D_{x}p_{2}.     
\end{align*}
So, using that $D_{x}(v(p))=D_{x}p\nabla^{2}\mathscr{K}(p)$ we get 
\begin{align}
(\partial_{t}+v_{1}\cdot \nabla)\eta(p_{1}\vert p_{2})&=(v_{1}-v_{2})\cdot \nabla V\ast (\rho_{2}-\rho_{1}) \notag\\
&+\nabla V\ast \rho_{2}\cdot (\nabla ^{2}\mathscr{K}(p_{2})(p_{1}-p_{2})-(\nabla \mathscr{K}(p_{1})-\nabla \mathscr{K}(p_{2}))) \notag\\
&-(v_{1}-v_{2})\cdot ((p_{1}-p_{2})D_{x}v_{2}). \label{material der of eta}
\end{align}
Hence, substituting \eqref{material der of eta} in  \eqref{first time der of K} yields  
\begin{align}
\frac{\dd}{\dd t}\mathcal{K}(t)=&2\int_{\mathbb{T}^{d}}\rho_{1}(v_{1}-v_{2})\cdot \nabla V\ast (\rho_{2}-\rho_{1})\ \dd x \notag\\
&+2\int_{\mathbb{T}^{d}}\rho_{1}\nabla V\ast \rho_{2}\cdot \left(\nabla^{2}\mathscr{K}(p_{2})(p_{1}-p_{2})-(\nabla \mathscr{K}(p_{1})-\mathscr{\nabla}\mathscr{K}(p_{2}))\right)\ \dd x \notag\\
&-2\int_{\mathbb{T}^{d}}\rho_{1}(v_{1}-v_{2})\cdot ((p_{1}-p_{2})D_{x}v_{2})\ \dd x. \label{Final calculation of evolution of kinetic part}   
\end{align}
\textbf{Time Variation of $\mathcal{V}$.} Using that $V$ is even we compute that 
\begin{align}
\frac{\dd}{\dd t}\mathcal{V}(t)&=2\int_{\mathbb{T}^{d}} V\ast (\rho_{1}-\rho_{2})\partial_{t}(\rho_{1}-\rho_{2})\ \dd x \notag\\ 
&=-2\int_{\mathbb{T}^{d}} V\ast (\rho_{1}-\rho_{2})\mathrm{div}(\rho_{1}v_{1}-\rho_{2}v_{2})\ \dd x \notag\\
&=2\int_{\mathbb{T}^{d}} \nabla V\ast (\rho_{1}-\rho_{2})\cdot(\rho_{1}v_{1}-\rho_{2}v_{2})\ \dd x \notag\\
&=2\int_{\mathbb{T}^{d}} \rho_{1}(v_{1}-v_{2})\cdot \nabla V\ast (\rho_{1}-\rho_{2})\ \dd x+2\int_{\mathbb{T}^{d}} v_{2}\cdot \nabla V\ast (\rho_{1}-\rho_{2})(\rho_{1}-\rho_{2})\ \dd x.  
\label{Claculation of variation of V}
\end{align}
Combining \eqref{Final calculation of evolution of kinetic part} and \eqref{Claculation of variation of V} we obtain 
\begin{align*}
\frac{\dd}{\dd t}\mathcal{E}(t)=&2\int_{\mathbb{T}^{d}} v_{2}\cdot \nabla V\ast (\rho_{1}-\rho_{2})(\rho_{1}-\rho_{2})\ \dd x\\
&+2\int_{\mathbb{T}^{d}}\rho_{1}(v_{1}-v_{2})\cdot ((p_{1}-p_{2})D_{x}v_{2})\ \dd x\\
&+2\int_{\mathbb{T}^{d}}\rho_{1}\nabla V\ast \rho_{2}\cdot \left(\nabla^{2}\mathscr{K}(p_{2})(p_{1}-p_{2})-(\nabla \mathscr{K}(p_{1})-\mathscr{\nabla}\mathscr{K}(p_{2}))\right)\ \dd x \coloneqq \sum_{k=1}^{3}I_{k}.     
\end{align*}
This establishes the identity \eqref{statement of formula for der of H}.\\
\textbf{Estimate on the  $I_{k}$.} In what follows the constant $C>0$ designates a constant depending only on $a,d,\left\Vert \rho_{2}\right\Vert_{\infty} ,\left\Vert p_{2}\right\Vert_{\infty}, \left\Vert D_{x}p_{2}\right\Vert_{\infty}$ and may vary between the different terms. 
To estimate $I_{1}$, note that 
\begin{align}
I_{1}&=-2\int_{\mathbb{T}^{d}} v_{2}\cdot \nabla V\ast (\rho_{1}-\rho_{2})(\rho_{1}-\rho_{2})\ \dd x \notag\\
&=2\int_{\mathbb{T}^{d}} v_{2}\cdot \nabla V\ast (\rho_{1}-\rho_{2})\mathrm{div}(\nabla V\ast (\rho_{1}-\rho_{2}))\ \dd x \notag\\
&=-2\int_{\mathbb{T}^{d}}D_{x}v_{2}\nabla V\ast(\rho_{1}-\rho_{2})\nabla V\ast (\rho_{1}-\rho_{2})\ \dd x-\int_{\mathbb{T}^{d}}v_{2}\cdot \nabla \left\vert \nabla V\ast(\rho_{1}-\rho_{2})\right\vert^{2}\ \dd x \notag
\\
&= -2\int_{\mathbb{T}^{d}}D_{x}v_{2}:\nabla V\ast (\rho_{1}-\rho_{2})\otimes \nabla V\ast (\rho_{1}-\rho_{2})\ \dd x+\int_{\mathbb{T}^{d}}\mathrm{div}(v_{2})\left\vert \nabla V\ast (\rho_{1}-\rho_{2})\right\vert^{2}\ \dd x \notag
\\ &\leq 3\left\Vert D_{x}v_{2}\right\Vert_{\infty} \mathcal{V}(t).    \label{I1EST MODULATED ENERGY}
\end{align}
To estimate $I_{2}$, we apply Lemma \ref{elementary ine lemma} in order to deduce   
\begin{align}
I_{2}&=-\int_{\mathbb{T}^
{d}} \rho_{1}(v_{1}-v_{2})\cdot((p_{1}-p_{2})D_{x}v_{2})\ \dd x \notag\\
&\leq \left\Vert D_{x}v_{2}\right\Vert_{\infty}\int_{\mathbb{T}^{d}}\rho_{1}\left\vert v_{1}-v_{2}\right\vert\left\vert p_{1}-p_{2}\right\vert\ \dd x
\leq C\left\Vert D_{x}v_{2}\right\Vert_{\infty} \int_{\mathbb{T}^{d}}\rho_{1}\frac{\left\vert p_{1}-p_{2}\right\vert^{2}}{1+\left\vert p_{1}-p_{2}\right\vert}\ \dd x. \notag 
\end{align}
Hence, by means of Lemma \ref{coercivity positivity of eta} we deduce 
\begin{align}
I_{2}\leq C\int_{\mathbb{T}^{d}}\rho_{1}\eta(p_{1}\vert p_{2})\ \dd x=C\mathcal{K}(t). \label{I2 est weak strong}      
\end{align}
To estimate $I_{3}$, we apply Lemma \ref{elementary ine lemma} and Young's inequality for convolutions to obtain 
\begin{align*}
I_{3}&\leq \left\Vert \nabla V\ast \rho_{2}\right\Vert_{\infty} \int_{\mathbb{T}^{d}}\rho_{1}\left\vert \nabla \mathscr{K}(p_{1})-\nabla \mathscr{K}(p_{2})-\nabla^{2}\mathscr{K}(p_{2})(p_{1}-p_{2})\right\vert\ \dd x\\
&\leq C\left\Vert \nabla V\right\Vert_{a'}\left\Vert \rho_{2}\right\Vert_{L^{\infty}L^{a}}\int_{\mathbb{T}^{d}} \rho_{1}\frac{\left\vert p_{1}-p_{2}\right\vert^{2}}{1+\left\vert p_{1}-p_{2}\right\vert}\ \dd x\leq C\int_{\mathbb{T}^{d}}\rho_{1}\frac{\left\vert p_{1}-p_{2}\right\vert^{2}}{1+\left\vert p_{1}-p_{2}\right\vert}\ \dd x. 
\end{align*}
Note that we used that $\nabla V\in L^{a'}(\mathbb{T}^{d})$ whenever $a>d$. 
Hence, by means of Lemma \ref{coercivity positivity of eta} we conclude that 
\begin{align}
I_{3}\leq C\int_{\mathbb{T}^{d}}\rho_{1}\eta(p_{1}\vert p_{2})\ \dd x\leq C\mathcal{K}(t).  
\label{I3 EST weak strong}     
\end{align}
Combining \eqref{I1EST MODULATED ENERGY}-\eqref{I3 EST weak strong} we conclude that 
\begin{align*}
\frac{\dd}{\dd t}\mathcal{E}(t)\leq C\mathcal{E}(t) \ \mbox{for some}\ C=C(a,d,\left\Vert D_{x}p_{2}\right\Vert_{\infty},\left\Vert p_{2}\right\Vert_{\infty},\left\Vert \rho_{2}\right\Vert_{L^{\infty}L^{a}}).    
\end{align*}
By Gr\"onwall's lemma and recalling that $\mathcal{E}(t)\geq 0$ we deduce the estimate \eqref{statement of est on H}. 
\qed
\section{Mean-field limit for relativistic $N$-body dynamics: Proof of Theorem \ref{main thm}}\label{Mean field limit section}
The general strategy underpinning the proof of the mean-field limit described in Theorem \ref{main thm} is to renormalize the argument presented in the previous section and take advantage of the functional inequality in Proposition \ref{Commutator estimates}. The following proposition together with Lemma \eqref{coercivity positivity of eta} guarantees that $\mathcal{E}_{N}(t)$ is non-negative.  
\begin{prop}[\cite{duerinckx2020mean},\,Corollary 3.5]
Suppose that: 
\begin{itemize}
\item[{\rm(i)}]$\rho \in L^{\infty}(\mathbb{T}^{d});$ 
   \item[{\rm(ii)}]
   $\mathbf{X}_{N}=(X_{1},\cdots\!,X_{N})\in \Delta_{N}^{c}$ and $\rho_{N}=\frac{1}{N}\sum_{i=1}^{N}\delta_{X_{i}}$. 
\end{itemize}
Let $\mathcal{V}_{N}(\rho_{N},\rho)$ be given by \eqref{def of VN}. Then, it holds that 
\begin{align*}
\mathcal{V}_{N}(\rho_{N},\rho) +\mathcal{C}_{N}\geq 0 \ \mbox{with}\ \mathcal{C}_{N}=\frac{(1+\left\Vert \rho\right\Vert_{\infty})\log N}{N}\mathbf{1}_{d=2}+\frac{1+\left\Vert \rho\right\Vert_{\infty} }{N^{\frac{2}{d}}}\mathbf{1}_{d\geq 3}. 
\end{align*} 
\label{non negativity}
\end{prop}
In addition we will use the following proposition in the final stage of the proof in order to deduce the convergence $\rho_{N}(t,\cdot)\underset{N\rightarrow \infty}{\rightharpoonup}\rho(t,\cdot)$ from the vanishing of $\mathcal{E}_{N}(t)$ as $N\rightarrow \infty$. 
\begin{prop}
\textup{\cite[Proposition 3.6]{duerinckx2020mean}}
Let the assumptions and notations of Proposition \ref{non negativity} hold. 
There are constants $C=C(d)>0$ and $\lambda=\lambda(d)>0$ such that for any $\varphi\in C^{\infty}(\mathbb{T}^{d})$ it holds that 
 \begin{align*}
&\int_{\mathbb{T}^{d}}\varphi(x)\left(\rho_{N}-\rho\right)(\dd x)\\
&\leq CN^{-\lambda}\left\Vert \nabla \varphi\right\Vert_{\infty} +C\left\Vert \nabla\varphi\right\Vert_{2} \left(\mathcal{V}_{N}(\rho_{N},\rho)+\frac{C(1+\left\Vert \rho \right\Vert_{\infty})}{N^{\frac{2}{d}}}\mathbf{1}_{d=3}+\frac{C(1+\left\Vert\rho \right\Vert_{\infty})\log N}{N}\mathbf{1}_{d=2}\right)^\frac{1}{2}.
 \end{align*}
 \label{coercivity inequality}
\end{prop}
We split the proof of Theorem \ref{main thm} into three steps. First, we prove the formula for the time variation of $\mathcal{E}_{N}(t)$ (identity \eqref{time der of renormalized relativistic modulated energy intro}), then we prove the evolution estimate \eqref{Gronwall est statement} and we conclude with the proof of the mean-field convergence \eqref{convergence as N  infty statement}. Each one of these steps reflects nontrivial adjustments in comparison to the non-relativistic settings.  

\vspace{0.4 cm}

\textbf{Proof of \ref{time der of renormalized relativistic modulated energy intro}.}
\textbf{Calculation of $\frac{\dd}{\dd t}\mathcal{K}_{N}(t)$.}
To make the equations lighter we omit the time variable whenever there is no ambiguity. We compute that 
\begin{align}
\frac{\dd}{\dd t}\mathcal{K}_{N}(t)&=\frac{\dd}{\dd t}\frac{2}{N}\sum_{i=1}^{N}(\mathscr{K}(P_{i})-\mathscr{K}(p(X_{i}))-\nabla \mathscr{K}(p(X_{i}))\cdot(P_{i}-p(X_{i}))) \notag\\
&=\frac{2}{N}\sum_{i=1}^{N}(\nabla \mathscr{K}(P_{i})\cdot \dot{P}_{i}(t)-\nabla\mathscr{K}(p(X_{i}))\cdot (\partial_{t}p(X_{i})+\dot X_{i}D_{x}p(X_{i}))) \notag\\
&-\frac{2}{N}\sum_{i=1}^{N}(\nabla ^{2}\mathscr{K}(p(X_{i}))(\partial_{t}p(X_{i})+\dot{X}_{i}D_{x}p(X_{i})))\cdot (P_{i}-p(X_{i})) \notag\\
&-\frac{2}{N}\sum_{i=1}^{N}\nabla \mathscr{K}(p(X_{i}))\cdot (\dot{P}_{i}-(\partial_{t}p(X_{i})+\dot{X}_{i}D_{x}p(X_{i}))). \label{right hand side of first kineticpart}
\end{align}
The first term in the right-hand side of \eqref{right hand side of first kineticpart} writes 
\begin{align*}
&\frac{2}{N}\sum_{i=1}^{N}\left(-\nabla \mathscr{K}(P_{i})\cdot \frac{1}{N}\sum_{j:j\neq i}\nabla V(X_{i}-X_{j})-\nabla \mathscr{K}(p(X_{i}))\cdot(\partial_{t}p(X_{i})+v(P_{i})D_{x}p(X_{i}))\right)\\
=&-\frac{2}{N^{2}}\sum_{i\neq j}\nabla \mathscr{K}(P_{i})\cdot \nabla V(X_{i}-X_{j})-\frac{2}{N}\sum_{i=1}^{N}\nabla \mathscr{K}(p(X_{i}))\cdot (\partial_{t}p(X_{i})+v(p(X_{i}))D_{x}p(X_{i}))\\
&-\frac{2}{N}\sum_{i=1}^{N}\nabla \mathscr{K}(p(X_{i}))\cdot (v(P_{i})-v(p(X_{i})))D_{x}p(X_{i})\\
=& \frac{2}{N}\sum_{i=1}^{N}\nabla \mathscr{K}(p(X_{i}))\cdot \nabla V\ast \rho(X_{i})-\frac{2}{N^{2}}\sum_{i\neq j}\nabla \mathscr{K}(P_{i})\cdot \nabla V(X_{i}-X_{j})\\
&-\frac{2}{N}\sum_{i=1}^{N}\nabla \mathscr{K}(p(X_{i}))\cdot (v(P_{i})-v(p(X_{i})))D_{x}p(X_{i})\coloneqq T_{1}^{1}+T_{2}^{1}+T_{3}^{1}. 
\end{align*}
The second term in the right-hand side of \eqref{right hand side of first kineticpart} writes 
\begin{align*}
&-\frac{2}{N}\sum_{i=1}^{N}\left(\nabla^{2}\mathscr{K}(p(X_{i}))(\partial_{t}p(X_{i})+v(P_{i})D_{x}p(X_{i}))\right)\cdot (P_{i}-p(X_{i}))\\
&=-\frac{2}{N}\sum_{i=1}^{N}\nabla^{2}\mathscr{K}(p(X_{i}))(v(P_{i})-v(p(X_{i})))D_{x}p(X_{i})\cdot (P_{i}-p(X_{i}))\\
&-\frac{2}{N}\sum_{i=1}^{N}\nabla^{2}\mathscr{K}(p(X_{i}))(\partial_{t}p(X_{i})+v(p(X_{i}))D_{x}p(X_{i}))\cdot(P_{i}-p(X_{i}))\\
&=\frac{2}{N}\sum_{i=1}^{N}(\nabla^{2}\mathscr{K}(p(X_{i}))\nabla V\ast \rho(X_{i}))\cdot (P_{i}-p(X_{i}))\\
&-\frac{2}{N}\sum_{i=1}^{N}\nabla^{2} \mathscr{K}(p(X_{i}))(v(P_{i})-v(p(X_{i})))D_{x}p(X_{i})\cdot (P_{i}-p(X_{i}))\\
&=\frac{2}{N}\sum_{i=1}^{N}(\nabla^{2}\mathscr{K}(p(X_{i}))\nabla V\ast \rho(X_{i}))\cdot (P_{i}-p(X_{i}))\\
&-\frac{2}{N}\sum_{i=1}^{N}(v(P_{i})-v(p(X_{i})))\cdot \left((P_{i}-p(X_{i}))D_{x}(v(p))(X_{i})\right)\coloneqq T^{2}_{1}+T^{2}_{2}. 
\end{align*}
The third term in the right-hand side of \eqref{right hand side of first kineticpart} is 
\begin{align*}
&-\frac{2}{N}\sum_{i=1}^{N}\nabla \mathscr{K}(p(X_{i}))\cdot \left(-\frac{1}{N}\sum_{j:j\neq i}\nabla V(X_{i}-X_{j})-(\partial_{t}p(X_{i})+v(P_{i})D_{x}p(X_{i}))\right)\\
&=\frac{2}{N^{2}}\sum_{i\neq j}\nabla \mathscr{K}(p(X_{i}))\cdot \nabla V(X_{i}-X_{j})+\frac{2}{N}\sum_{i=1}^{N}\nabla \mathscr{K}(p(X_{i}))\cdot \left((v(P_{i})-v(p(X_{i})))D_{x}p(X_{i})\right)\\
&+\frac{2}{N}\sum_{i=1}^{N}\nabla \mathscr{K}(p(X_{i}))\cdot (\partial_{t}p(X_{i})+v(p(X_{i}))D_{x}p(X_{i}))\\
&=-\frac{2}{N}\sum_{i=1}^{N}\nabla \mathscr{K}(p(X_{i}))\cdot \nabla V\ast \rho(X_{i})+\frac{2}{N^{2}}\sum_{i\neq j}\nabla \mathscr{K}(p(X_{i}))\cdot \nabla V(X_{i}-X_{j})\\
&+\frac{2}{N}\sum_{i=1}^{N}\nabla \mathscr{K}(p(X_{i}))\cdot (\left(v(P_{i})-v(p(X_{i}))\right)D_{x}p(X_{i}))
\\
&=-\frac{2}{N}\sum_{i=1}^{N}\nabla \mathscr{K}(P_{i})\cdot \nabla V\ast \rho(X_{i})-\frac{2}{N}\sum_{i=1}^{N}(\nabla \mathscr{K}(p(X_{i}))-\nabla \mathscr{K}(P_{i}))\cdot \nabla V\ast \rho(X_{i})\\
&+\frac{2}{N^{2}}\sum_{i\neq j}\nabla \mathscr{K}(p(X_{i}))\cdot \nabla V(X_{i}-X_{j})+\frac{2}{N}\sum_{i=1}^{N}\nabla \mathscr{K}(p(X_{i}))\cdot \left((v(P_{i})-v(p(X_{i})))D_{x}p(X_{i})\right)
\\&= T^{3}_{1}+T^{3}_{2}+T^{3}_{3}+T^{3}_{4}. 
\end{align*}
We compute that 
\begin{align}
&T^{1}_{1}+T^{2}_{1}+T^{3}_{1} \notag\\
&=  \frac{2}{N}\sum_{i=1}^{N}\nabla \mathscr{K}(p(X_{i}))\cdot \nabla V\ast \rho(X_{i}) \notag\\
&+\frac{2}{N}\sum_{i=1}^{N}(\nabla^{2}\mathscr{K}(p(X_{i}))\nabla V\ast \rho(X_{i}))(P_{i}-p(X_{i}))-\frac{2}{N}\sum_{i=1}^{N}\nabla \mathscr{K}(P_{i})\cdot \nabla V\ast \rho(X_{i}) \notag\\
&=\frac{2}{N}\sum_{i=1}^{N}\left(\nabla^{2} \mathscr{K}(p(X_{i}))(P_{i}-p(X_{i}))-(\nabla \mathscr{K}(P_{i})-\nabla \mathscr{K}(p(X_{i})))\right)\cdot \nabla V\ast \rho(X_{i}). \label{Calculation of T11 T21 T31}
\end{align}
In addition, we compute that  
\begin{align}
&T^{1}_{2}+T^{3}_{2}+T^{3}_{3} \notag\\
&=-\frac{2}{N^
{2}}\sum_{i\neq j}\nabla \mathscr{K}(P_{i})\cdot \nabla V(X_{i}-X_{j}) \notag\\
&-\frac{2}{N}\sum_{i=1}^{N}(\nabla \mathscr{K}(p(X_{i}))-\nabla \mathscr{K}(P_{i}))\cdot \nabla V\ast \rho(X_{i}) \notag\\
&+\frac{2}{N^{2}}\sum_{i\neq j}\nabla \mathscr{K}(p(X_{i}))\cdot \nabla V(X_{i}-X_{j}) \notag\\
&=\frac{2}{N}\sum_{i=1}^{N}(\nabla \mathscr{K}(P_{i})-\nabla \mathscr{K}(p(X_{i})))\cdot \left(\nabla V\ast \rho(X_{i})-\frac{1}{N}\sum_{j:j\neq i}\nabla V(X_{i}-X_{j})\right)\notag\\
&=\frac{2}{N}\sum_{i=1}^{N}(v(P_{i})-v(p(X_{i})))\cdot \left(\nabla V\ast \rho(X_{i})-\frac{1}{N}\sum_{j:j\neq i}\nabla V(X_{i}-X_{j})\right). 
\end{align}
Moreover, we have 
\begin{align}
T_{3}^{1}+T^{2}_{2}+T^{3}_{4}&=  -\frac{2}{N}\sum_{i=1}^{N}\nabla \mathscr{K}(p(X_{i}))\cdot \left((v(P_{i})-v(p(X_{i})))D_{x}p(X_{i})\right)\notag\\
&-\frac{2}{N}\sum_{i=1}^{N}(v(P_{i})-v(p(X_{i})))\left((P_{i}-p(X_{i}))D_{x}(v(p))(X_{i})\right) \notag\\
&+\frac{2}{N}\sum_{i=1}^{N}\nabla \mathscr{K}(p(X_{i}))\cdot \left((v(P_{i})-v(p(X_{i})))D_{x}p(X_{i})\right) \notag\\
&=-\frac{2}{N}\sum_{i=1}^{N}(v(P_{i})-v(p(X_{i})))(P_{i}-p(X_{i}))D_{x}(v(p))(X_{i}).
\label{Calculation T13T22T34} 
\end{align}
Gathering \eqref{Calculation of T11 T21 T31}-\eqref{Calculation T13T22T34} we obtain 
\begin{align}
\frac{\dd}{\dd t}\mathcal{K}_{N}(t)&= \frac{2}{N}\sum_{i=1}^{N}(v(P_{i})-v(p(X_{i})))\cdot \left(\nabla V\ast \rho(X_{i})-\frac{1}{N}\sum_{j:j\neq i}\nabla V(X_{i}-X_{j})\right)   
 \notag\\
&+\frac{2}{N}\sum_{i=1}^{N}\left(\nabla^{2} \mathscr{K}(p(X_{i}))(P_{i}-p(X_{i}))-(\nabla \mathscr{K}(P_{i})-\nabla \mathscr{K}(p(X_{i})))\right)\cdot \nabla V\ast \rho(X_{i}) \notag\\
&-\frac{2}{N}\sum_{i=1}^{N}(v(P_{i})-v(p(X_{i})))\left((P_{i}-p(X_{i}))D_{x}(v(p))(X_{i})\right). \label{time der of KN} 
\end{align}
\textbf{Calculation of $\frac{\dd}{\dd t}\mathcal{V}_{N}(t)$.} We have 
\begin{align*}
\frac{\dd}{\dd t}\mathcal{V}_{N}(t)=\frac{\dd}{\dd t}\left(\frac{1}{N^{2}}\sum_{i\neq j}V(X_{i}-X_{j})-\frac{2}{N}\sum_{i=1}^{N}V\ast \rho(t,X_{i})+\int_{\mathbb{T}^{d}}\rho V\ast\rho \ \dd x \right).    
\end{align*}
First, note that 
\begin{align}
\frac{\dd}{\dd t}\left(\frac{1}{N^{2}}\sum_{i\neq j}V(X_{i}-X_{j})\right)=\frac{1}{N^{2}}\sum_{i\neq j}\nabla V(X_{i}-X_{j})\cdot(v(P_{i})-v(P_{j})). \label{time der of pure product of empirical}     
\end{align}
In addition, we compute \begin{align}
&-\frac{\dd}{\dd t}\frac{2}{N}\sum_{i=1}^{N}V\ast \rho(X_{i}) \notag\\
&=-\frac{2}{N}\sum_{i=1}^{N} V\ast \partial_{t}\rho(X_{i})-\frac{2}{N}\sum_{i=1}^{N} \nabla V\ast \rho(X_{i})\cdot v(P_{i}) \notag\\
&=\frac{2}{N}\sum_{i=1}^{N}\nabla V\ast (\rho v(p))(X_{i})-\frac{2}{N}\sum_{i=1}^{N}\nabla V\ast \rho(X_{i})\cdot v(P_{i}) \notag\\
&=\frac{2}{N}\sum_{i=1}^{N}\int_{\mathbb{T}^{d}}\nabla V(X_{i}-y)\cdot \rho v(p(y))\ \dd y-  \frac{2}{N}\sum_{i=1}^{N}\nabla V\ast \rho(X_{i})\cdot v(P_{i}) \notag\\
&=2\int_{\mathbb{T}^{d}\times \mathbb{T}^{d}}\nabla V(x-y)\cdot \rho(y)v(p(y))\ \dd y \rho_{N}(\dd x)-  \frac{2}{N}\sum_{i=1}^{N}\nabla V\ast \rho(X_{i})\cdot v(P_{i}). \label{time der of mixed products}
\end{align}
Moreover, using that $V$ is even we have 
\begin{align}
\frac{\dd}{\dd t}\int_{\mathbb{T}^{d}}\rho V\ast \rho\ \dd x&=2\int_{\mathbb{T}^{d}}\partial_{t}\rho V\ast \rho \ \dd x \notag\\
&=-2\int_{\mathbb{T}^{d}}\mathrm{div}(\rho v(p))V\ast \rho \ \dd x=2\int_{\mathbb{T}^{d}} \rho v(p)\cdot \nabla V\ast \rho \ \dd x \notag\\
&=2\int_{\mathbb{T}^{d}\times \mathbb{T}^{d}} v(p(x))\cdot \nabla V(x-y)\rho(y)\rho(x)\ \dd x\dd y \notag\\
&=\int_{\mathbb{T}^{d}\times \mathbb{T}^{d}}(v(p(x))-v(p(y)))\cdot \nabla V(x-y)\rho(x)\rho(y)\ \dd x\dd y.   \label{time der of pure product of limit}
\end{align}
To conclude, gathering \eqref{time der of pure product of empirical}-\eqref{time der of pure product of limit} we arrive at the following identity 
\begin{align}
\frac{\dd}{\dd t}\mathcal{V}_{N}(t)=
&-\int_{\Delta^{c}}(v(p(x))-v(p(y)))\cdot \nabla V(x-y)\rho_{N}^{\otimes 2}(\dd x\dd y) \notag\\
&+2\int_{\mathbb{T}^{d}\times \mathbb{T}^{d}}(v(p(x))-v(p(y)))\cdot \nabla V(x-y)\rho(x)\ \dd x\rho_{N}(\dd y) \notag\\
&+2\int_{\mathbb{T}^{d}\times \mathbb{T}^{d}}\nabla V(x-y)\cdot v(p(y))\rho(y)\ \dd y\rho_{N}(\dd x)-\frac{2}{N}\sum_{i=1}^{N}\nabla V\ast \rho(t,X_{i})\cdot v(P_{i}) \notag\\
&+\frac{1}{N^{2}}\sum_{i\neq j}\nabla V(X_{i}-X_{j})\cdot (v(P_{i})-v(P_{j})) \notag\\
&+\int_{\Delta^{c}}(v(p(x))-v(p(y)))\cdot \nabla V(x-y)(\rho_{N}-\rho)^{\otimes 2}(\dd x\dd y)\coloneqq\sum_{k=1}^{6}L_{k}. 
\label{sum of the Lk}
\end{align}
\textbf{Conclusion.} 
First, we compute 
\begin{align*}
 &\frac{2}{N}\sum_{i=1}^{N}(v(P_{i})-v(p(X_{i})))\cdot \left(\nabla V\ast \rho(X_{i})-\frac{1}{N}\sum_{j:j\neq i}\nabla V(X_{i}-X_{j})\right) \\
 &=\frac{2}{N}\sum_{i=1}^{N}v(P_{i})\cdot \nabla V\ast \rho(X_{i})-\frac{2}{N^{2}}\sum_{i\neq j}v(P_{i})\cdot \nabla V(X_{i}-X_{j})\\
 &-\frac{2}{N}\sum_{i=1}^
 {N}v(p(X_{i}))\cdot \nabla V\ast \rho(X_{i})+\frac{2}{N^{2}}\sum_{i\neq j}v(p(X_{i}))\cdot \nabla V(X_{i}-X_{j})\coloneqq\sum_{k=1}^{4}M_{k}.  
\end{align*}
Since $\nabla V$ is odd we observe that 
\begin{align*}
\frac{1}{N^{2}}\sum_{i\neq j}\nabla V(X_{i}-X_{j})\cdot (v(P_{i})-v(P_{j}))=\frac{2}{N^{2}}\sum_{i\neq j}v(P_{i})\cdot \nabla V(X_{i}-X_{j})
\end{align*}
and so  
\begin{align}
L_{4}+L_{5}+M_{1}+M_{2}=0.  \label{first vanishing}   
\end{align}
In addition, by symmetrizing we obtain  
\begin{align*}
M_{4}=2\int_{\Delta^{c}}v(p(x))\cdot \nabla V(x-y)\rho_{N}^{\otimes 2}(\dd x\dd y)=\int_{\Delta^{c}}\left(v(p(x))-v(p(y))\right)\cdot \nabla V(x-y)\rho_{N}^{\otimes 2}(\dd x\dd y),   
\end{align*}
and so 
\begin{align}
M_{4}+L_{1}=0. \label{second vanishing}    
\end{align}
Furthermore, we compute that  
\begin{align*}
L_{2}+L_{3}=&-\frac{2}{N}\sum_{i=1}^{N} \nabla V\ast (\rho v(p))(X_{i})+\frac{2}{N}\sum_{i=1}^{N}\nabla V\ast (\rho v(p))(X_{i})\\
&+\frac{2}{N}\sum_{i=1}^{N}v(p(X_{i}))\cdot\nabla V\ast \rho(X_{i})=\frac{2}{N}\sum_{i=1}^{N}v(p(X_{i}))\cdot \nabla V \ast \rho(X_{i}),     
\end{align*}
so that 
\begin{align}
L_{2}+L_{3}+M_{3}=0. \label{third vanishing} \end{align}
In view of \eqref{time der of KN}, \eqref{sum of the Lk} and \eqref{first vanishing}-\eqref{third vanishing} we arrive at the identity 
\begin{align}
\frac{\dd}{\dd t}\mathcal{E}_{N}(t)&=\int_{\mathbb{T}^{d}\times \mathbb{T}^{d}}(v(p(x))-v(p(y)))\cdot \nabla V(x-y)(\rho_{N}-\rho)^{\otimes 2}(\dd x\dd y) \notag\\
&-\frac{2}{N}\sum_{i=1}^{N}(v(P_{i})-v(p(X_{i})))\left((P_{i}-p(X_{i}))D_{x}(v(p))(X_{i})\right) \notag\\
&+\frac{2}{N}\sum_{i=1}^{N}(\nabla^{2}\mathscr{K}(p(X_{i}))(P_{i}-p(X_{i}))-(\nabla \mathscr{K}(P_{i})-\nabla \mathscr{K}(p(X_{i}))))\cdot \nabla V\ast \rho(X_{i}) \notag\\
&\coloneqq I_{1}+I_{2}+I_{3}. \label{I1I3 DEF} 
\end{align}
\qed

\vspace{0.4 cm}

\textbf{Proof of \eqref{Gronwall est statement}.} To prove \eqref{Gronwall est statement} we estimate separately each one of the $I_{k}$ (as defined in \eqref{I1I3 DEF}). By Proposition \ref{commutator estimate} there is a constant $C=C(\left\Vert D_{x}p\right\Vert_{\infty},\left\Vert \rho\right\Vert_{\infty})>0$ such that   
\begin{align}
I_{1}\leq C\left(\mathcal{V}_{N}(t)+\frac{\log N}{N}\mathbf{1}_{d=2}+\frac{1}{N^{\frac{2}{d}}}\mathbf{1}_{d\geq3}\right). 
\label{Estimate I1 final}
\end{align}
To estimate $I_{2}$ we apply inequality \eqref{first preliminary ine} in Lemma \ref{elementary ine lemma} to infer 
\begin{align*}
I_{2}\leq \frac{2C\left\Vert D_{x}p\right\Vert_{\infty}}{N}\sum_{i=1}^{N}\left\vert v(P_{i})-v(p(X_{i}))\right\vert \left\vert P_{i}-p(X_{i})\right\vert\leq \frac{C}{N}\sum_{i=1}^{N}\frac{\left\vert P_{i}-p(X_{i})\right\vert^{2} }{1+\left\vert P_{i}-p(X_{i})\right\vert}.        
\end{align*}
In view of Lemma \ref{coercivity positivity of eta} it follows that 
\begin{align}
I_{2}\leq \frac{C}{N}\sum_{i=1}^{N}\eta(P_{i}\vert p(X_{i})). \label{I2 mean field est}     
\end{align}
To estimate $I_{3}$, we invoke Lemma \ref{elementary ine lemma} to deduce that 
\begin{align*}
I_{3}&\leq \frac{2\left\Vert \nabla V\ast \rho\right\Vert_{\infty} }{N}\sum_{i=1}^{N}\left\vert \nabla^{2} \mathscr{K}(p(X_{i}))(P_{i}-p(X_{i}))-(\nabla \mathscr{K}(P_{i})-\mathscr{K}(p(X_{i})))\right\vert\\
&\leq \frac{2C\left\Vert \nabla V\right\Vert_{a'} \left\Vert \rho\right\Vert_{L^{\infty}L^{a}
}}{N}\sum_{i=1}^{N}\frac{\left\vert P_{i}-p(X_{i})\right\vert^{2} }{1+\left\vert P_{i}-p(X_{i})\right\vert }.     
\end{align*}
Hence, by means of Lemma \ref{coercivity positivity of eta} we infer that 
\begin{align}
I_{3}\leq \frac{C}{N}\sum_{i=1}^{N}\eta(P_{i}\vert p(X_{i})). \label{I3 mean field est}    
\end{align}
Combining \eqref{Estimate I1 final}-\eqref{I3 mean field est} we conclude that 
\begin{align*}
\frac{\dd}{\dd t}\mathcal{E}_{N}(t)\leq C(\mathcal{E}_{N}(t)+o_{N}(1)).    
\end{align*}
Thanks to Gr\"onwall's lemma we establish \eqref{Gronwall est statement}. 
\qed 

\vspace{0.4 cm}

\textbf{Proof of \eqref{convergence as N infty statement}.} That
\begin{align}
\underset{t\in [0,T]}{\sup}W_{1}(\rho_{N}(t,\cdot),\rho(t,\cdot))\underset{N\rightarrow \infty}{\rightarrow} 0  \label{Wasserstein ditance between rhoN and rho}
\end{align}
follows directly from Proposition \ref{coercivity inequality} and the estimate \eqref{Gronwall est statement}. We proceed by showing that $$\underset{t\in [0,T]}{\sup}\left\Vert (q_{N}-\rho p)(t,\cdot)\right\Vert_{W^{-1,\infty}}\underset{N\rightarrow \infty}{\rightarrow} 0.$$ Given $b\in W^{1,\infty}(\mathbb{T}^{d};\mathbb{R}^{d})$ with $\left\Vert b\right\Vert_{W^{1,\infty}}\leq 1$ we have  
\begin{align*}
\int_{\mathbb{T}^{d}} b(x)\cdot (q_{N}-\rho p)(t,\dd x)&=\int_{\mathbb{T}^{d}}b(x)\cdot (q_{N}-\rho_{N}p)(t,\dd x) +\int_{\mathbb{T}^{d}}p\cdot b(x)(\rho_{N}-\rho)(t,\dd x)\\
&\coloneqq \mathcal{T}_{1,N}+\mathcal{T}_{2,N}.     
\end{align*}
Fix $\lambda_{N}>1$ to be chosen later. The term $\mathcal{T}_{1,N}$ is recast as 
\begin{align*}
\mathcal{T}_{1,N}&=\frac{1}{N}\sum_{i=1}^{N}(P_{i}-p(X_{i}))\cdot b(X_{i})\\
&=\frac{1}{N}\sum_{\left\vert P_{i}-p(X_{i})\right\vert\leq \frac{1}{\lambda_{N} -1}}(P_{i}-p(X_{i}))\cdot b(X_{i})+\frac{1}{N}\sum_{\left\vert P_{i}-p(X_{i})\right\vert\geq \frac{1}{\lambda_{N} -1}}(P_{i}-p(X_{i}))\cdot b(X_{i})\\
&\coloneqq  \Sigma_{1}+\Sigma_{2}.
\end{align*}
Note that as long as $\left\vert P_{i}-p(X_{i})\right\vert\geq \frac{1}{\lambda_{N} -1}$ we have 
\begin{align*}
\left\vert P_{i}-p(X_{i})\right\vert\leq \frac{\lambda_{N} \left\vert P_{i}-p(X_{i})\right\vert^{2}}{1+\left\vert P_{i}-p(X_{i})\right\vert}. 
\end{align*}
Therefore, by means of  Lemma \ref{coercivity positivity of eta} we obtain 
\begin{align*}
\Sigma_{2}&\leq \frac{\left\Vert b\right\Vert_{\infty} }{N}\sum_{\left\vert P_{i}-p(X_{i})\right\vert\geq \frac{1}{\lambda_{N}-1}}\left\vert P_{i}-p(X_{i})\right\vert\\
&\leq \frac{C}{N}\sum_{i=1}^{N}\frac{\lambda_{N} \left\vert P_{i}-p(X_{i})\right\vert^{2}}{1+\left\vert P_{i}-p(X_{i})\right\vert}\leq C\lambda_{N}\mathcal{K}_{N}(t)\leq C\lambda_{N}\mathcal{E}_{N}(t). 
\end{align*}
Let $o_{N}(1)$ be defined as in \eqref{Gronwall est statement} and choose $\lambda_{N}=\max\{\frac{1}{\sqrt{\mathcal{E}_{N}(0)}},\frac{1}{\sqrt{o_{N}(1)}}\}$. By \eqref{Gronwall est statement} we have $\underset{t\in [0,T]}{\sup}\mathcal{E}_{N}(t)\leq Ce^{CT}(\mathcal{E}_{N}(0)+To_{N}(1))$ and so we obtain
\begin{align*}
\Sigma_{2}\leq Ce^{CT}\left(\sqrt{\mathcal{E}_{N}(0)}+T\sqrt{o_{N}(1)}\right)\underset{N\rightarrow \infty}{\rightarrow}0.    
\end{align*}
Moreover we estimate 
\begin{align*}
\Sigma_{1}\leq \frac{\left\Vert b\right\Vert_{\infty}}{N}\sum_{\left\vert P_{i}-p(X_{i})\right\vert\leq \frac{1}{\lambda_{N}-1}}\left\vert P_{i}-p(X_{i})\right\vert\leq \frac{1}{(\lambda_{N}-1)N}\sum_{i=1}^{N}1=\frac{1}{\lambda_{N} -1}\underset{N\rightarrow \infty}{\rightarrow}0.      
\end{align*}
Thus, we infer that  
\begin{align}
\underset{t\in [0,T]}{\sup}\mathcal{T}_{1,N}(t)\underset{N\rightarrow \infty}{\rightarrow} 0. \label{T1 vanishing} \end{align}
As for $\mathcal{T}_{2,N}$, in view of \eqref{Wasserstein ditance between rhoN and rho} we have 
\begin{align}
\underset{t\in [0,T]}{\sup}\mathcal{T}_{2,N}(t)&\leq \left\Vert \nabla (p\cdot b)\right\Vert_{\infty}W_{1}(\rho_{N}(t,\cdot),\rho(t,\cdot)) \notag\\
&\leq \left\Vert p\right\Vert_{L^{\infty}W^{1,\infty}}\underset{t\in [0,T]}{\sup}W_{1}(\rho_{N}(t,\cdot),\rho(t,\cdot))\underset{N\rightarrow \infty}{\rightarrow} 0. \label{T2 vanishing}
\end{align}
To conclude, owing to \eqref{T1 vanishing}-\eqref{T2 vanishing} we deduce that  $$\underset{t\in [0,T]}{\sup}\left\Vert (q_{N}-\rho p)(t,\cdot)\right\Vert_{W^{-1,\infty}}\underset{N\rightarrow \infty}{\rightarrow} 0$$. 
\qed
\begin{rem}
The statistical relevance of initial data for the non-relativistic renormalized modulated energy $H_{N}$ 
(as defined in \eqref{nonrel renormalized modulated energy}) has been already demonstrated in \cite{han2021newton} by exploiting the large deviation estimates in \cite{fournier2015rate}. The statistical relevance for the relativistic renormalized modulated energy $\mathcal{E}_{N}$ follows by exactly the same argument.
\end{rem}
\begin{rem}
The existence of a globally well defined flow for \eqref{Relativistic N body dynamics} follows by observing separation of particles and does not necessitate any significant modifications in comparison to the non-relativistic settings (which have been considered for instance in \cite{porat2026supercritical}).       
\end{rem}
\noindent{\bf Acknowledgments.}
This work was prepared during the author's stay at the University of Texas at Austin. The author is indebted to the University of Texas at Austin for financial support.

\vspace{0.4 cm}

\noindent{\bf Conflict of Interest:} The author declares that he has no conflict of interest.
The author also declares that this manuscript has not been previously published,
and will not be submitted elsewhere before your decision.

\bigskip
\noindent{\bf Data availability:} Data sharing is not applicable to this article as no datasets were generated or analyzed during the current study.

\bigskip

\noindent{\bf Declaration on the Use of Generative AI.} The author acknowledges the use of ChatGPT as an auxiliary tool in exploring mathematical ideas and checking certain intermediate arguments and computations. All AI-generated suggestions were independently assessed and verified by the author. The mathematical results, proofs, organization, and conclusions of the paper are the author’s own, and the manuscript was written in its entirety by the author, who assumes full responsibility for its content.
\bibliographystyle{abbrv}
\bibliography{references}
\end{document}